\documentclass[12pt, reqno]{amsart}
\usepackage{amssymb, mathrsfs}
\usepackage{latexsym}
\usepackage{amsmath}
\usepackage{amsthm}
\usepackage{amsfonts}
\usepackage{dsfont, upgreek}
\usepackage{mathtools}
\usepackage{epsfig}
\usepackage[scr]{rsfso}
\usepackage[nocompress]{cite}
\usepackage{amscd}
\usepackage{graphicx}
\usepackage{tikz-cd}

\usepackage{enumerate}

\usepackage{todonotes}

\usepackage{color}
\usepackage{tikz}
\usetikzlibrary{patterns}

\usepackage{geometry}
\usetikzlibrary{decorations.markings}
\usetikzlibrary{arrows.meta}

\usepackage{extarrows}

\usepackage{adjustbox}
\usepackage{pgfplots}
\usepgfplotslibrary{colormaps}
\usepackage{slashed}

\usepackage[colorlinks=true,linkcolor=blue,citecolor=blue]{hyperref}

\usepackage{geometry}
\usepackage[msc-links, lite]{amsrefs}

\usepackage[all,cmtip]{xy}

\theoremstyle{plain}
\newtheorem{theorem}{Theorem}[section]
\newtheorem{proposition}[theorem]{Proposition}
\newtheorem{lemma}[theorem]{Lemma}
\newtheorem{corollary}[theorem]{Corollary}

\newtheorem{conjecture}[theorem]{Conjecture}
\newtheorem{problem}[theorem]{Problem}
\theoremstyle{definition} 
\newtheorem{definition}[theorem]{Definition}

\newtheorem{example}[theorem]{Example}
\newtheorem*{claim*}{Claim}

\theoremstyle{remark} 
\newtheorem{remark}[theorem]{Remark}

\numberwithin{equation}{section}

\newcommand{\Sc}{\mathrm{Sc}}

\newcommand{\Z}{\mathbb{Z}}

\newcommand{\R}{\mathbb{R}}
\newcommand{\dist}{\mathrm{dist}}

\newcommand{\ind}{\textup{Ind}}

\newcommand{\supp}{\mathrm{supp}}

\newcommand{\prop}{\mathrm{prop}}

\newcommand{\sph}{\mathbb{S}}

\newcommand{\Ric}{\mathrm{Ric}}

\newcommand{\ev}{ev}

\DeclareMathOperator{\inj}{Inj}

\newcommand{\interior}[1]{%
	{\kern0pt#1}^{\mathrm{\,o}}%
}

\makeatletter
\let\save@mathaccent\mathaccent
\newcommand*\if@single[3]{%
	\setbox0\hbox{${\mathaccent"0362{#1}}^H$}%
	\setbox2\hbox{${\mathaccent"0362{\kern0pt#1}}^H$}%
	\ifdim\ht0=\ht2 #3\else #2\fi
}
\newcommand*\rel@kern[1]{\kern#1\dimexpr\macc@kerna}
\newcommand*\wideaccent[2]{\@ifnextchar^{{\wide@accent{#1}{#2}{0}}}{\wide@accent{#1}{#2}{1}}}
\newcommand*\wide@accent[3]{\if@single{#2}{\wide@accent@{#1}{#2}{#3}{1}}{\wide@accent@{#1}{#2}{#3}{2}}}
\newcommand*\wide@accent@[4]{%
	\begingroup
	\def\mathaccent##1##2{%
		\let\mathaccent\save@mathaccent
		\if#42 \let\macc@nucleus\first@char \fi
		\setbox\z@\hbox{$\macc@style{\macc@nucleus}_{}$}%
		\setbox\tw@\hbox{$\macc@style{\macc@nucleus}{}_{}$}%
		\dimen@\wd\tw@
		\advance\dimen@-\wd\z@
		\divide\dimen@ 3
		\@tempdima\wd\tw@
		\advance\@tempdima-\scriptspace
		\divide\@tempdima 10
		\advance\dimen@-\@tempdima
		\ifdim\dimen@>\z@ \dimen@0pt\fi
		\rel@kern{0.6}\kern-\dimen@
		\if#41
		#1{\rel@kern{-0.6}\kern\dimen@\macc@nucleus\rel@kern{0.4}\kern\dimen@}%
		\advance\dimen@0.4\dimexpr\macc@kerna
		\let\final@kern#3%
		\ifdim\dimen@<\z@ \let\final@kern1\fi
		\if\final@kern1 \kern-\dimen@\fi
		\else
		#1{\rel@kern{-0.6}\kern\dimen@#2}%
		\fi
	}%
	\macc@depth\@ne
	\let\math@bgroup\@empty \let\math@egroup\macc@set@skewchar
	\mathsurround\z@ \frozen@everymath{\mathgroup\macc@group\relax}%
	\macc@set@skewchar\relax
	\let\mathaccentV\macc@nested@a
	\if#41
	\macc@nested@a\relax111{#2}%
	\else
	\def\gobble@till@marker##1\endmarker{}%
	\futurelet\first@char\gobble@till@marker#2\endmarker
	\ifcat\noexpand\first@char A\else
	\def\first@char{}%
	\fi
	\macc@nested@a\relax111{\first@char}%
	\fi
	\endgroup
}
\makeatother

\makeatletter
\newcommand*{\transpose}{%
	{\mathpalette\@transpose{}}%
}
\newcommand*{\@transpose}[2]{%
	\raisebox{\depth}{$\m@th#1\intercal$}%
}
\makeatother

\makeatletter
\providecommand\@dotsep{5}
\renewcommand{\listoftodos}[1][\@todonotes@todolistname]{%
	\@starttoc{tdo}{#1}}
\makeatother

\begin{document}
\title[Sharp bottom spectrum and scalar curvature rigidity]{Sharp bottom spectrum and scalar curvature rigidity}

\author{Jinmin Wang}
\address[Jinmin Wang]{Institute of Mathematics, Chinese Academy of Sciences}
\email{jinmin@amss.ac.cn}
\thanks{}

\author{Bo Zhu}
\address[Bo Zhu]{ Yau Mathematical Sciences Center,  Tsinghua University}
\email{zhub@tsinghua.edu.cn}
\thanks{}

\begin{abstract}
 We establish a sharp upper bound for the bottom spectrum of the Beltrami Laplacian on universal covers of closed Riemannian manifolds with a scalar curvature lower bound. Moreover, we prove a scalar curvature rigidity theorem when this bound is achieved. Additionally, we prove a net characterization of scalar curvature for general complete noncompact Riemannian manifolds.
\end{abstract}
\maketitle

\section{Introduction}

Suppose that  $(X^n,g)$ is  a connected,  complete Riemannian manifold and $\Delta$ is the corresponding Beltrami Laplacian on $(X^n,g)$  defined as
\begin{equation*} \label{def: laplace} 
	\Delta f = \sum_{i=1}^n \left(\nabla_{e_i}\nabla_{e_j}  - \nabla_{\nabla_{e_i} e_j}\right) f = \frac{1}{\sqrt{\det(g)}} \frac{\partial}{\partial x^i}\left(\sqrt{\det(g)}g^{ij}\frac{\partial f }{\partial x^j}\right).
\end{equation*}
The $L^2$-bottom spectrum of $\Delta$  on $(X,g)$ is defined by (see  \cite{Cheng_eigenvalue}*{Section 4} or \cite{Li_geometric_analysis}*{Definition 6.3})
\begin{equation}\label{def: lambda_1}
	\lambda_1(X,g) = \inf \left\{ \frac{\int_X |\nabla f|^2}{ \int_X f^2}: \ f \in C_c^\infty(X), \ f \neq 0\right\}.
\end{equation}

Using the classical comparison theorem (see \cite{Cheng_eigenvalue}*{Theorem~4.2}), Shiu-Yuen Cheng first proved that if a Riemannian manifold $(X^n, g)$ satisfies $\Ric_g \geq -(n-1)$, then
\begin{equation} \label{eq: cheng_spectrum}
    \lambda_1(X, g) \leq \frac{(n-1)^2}{4}.
\end{equation}
Following this, Peter Li and Jiaping Wang pioneered the use of harmonic function theory to study the sharpness of the upper bound in \eqref{eq: cheng_spectrum} and to investigate splitting rigidity phenomena on complete, noncompact Riemannian manifolds with Ricci curvature bounded below by $-(n-1)$ (see \cites{Li_Wang_positive_spectrum_1,Li_Wang_positive_spectrum_2}).

A natural question is whether the estimate \eqref{eq: cheng_spectrum} can be generalized to manifolds with a scalar curvature lower bound.  Munteanu--Wang recently extended the sharp bottom spectrum estimate to three-dimensional manifolds with a negative scalar curvature lower bound, again employing harmonic function techniques.

\begin{theorem}[Munteanu--Wang, see \cite{MunteanuWang}*{Theorem 1.1}] \label{thm: mw_bottom_spectrum}
  Suppose that  $(X^3,g)$ is a complete, noncompact, three-dimensional Riemannian manifold with scalar curvature $\Sc_g \geq -6$. If $X$ satisfies either one of the following properties:
\begin{itemize}
	\item the second homology group $H_2(X, \mathbb{Z})$ contains no spherical class, or
	\item $X$ has finitely many ends and finite first Betti number $b_1(X) < \infty$,
\end{itemize}
then \[\lambda_1(X,g) \leq \frac{(n-1)^2}{4}.\]  
\end{theorem}
Note that Theorem~\ref{thm: mw_bottom_spectrum} applies to the universal cover of any closed, three-dimensional aspherical manifold. Moreover, Munteanu--Wang point out that Theorem~\ref{thm: mw_bottom_spectrum} does not hold in general without any topological or geometric constraints (see \cite{MunteanuWang}*{Example~1.2}).

Using the $L^2$-index theory of Dirac operators on spin manifolds, H. Davoux proved a similar result for a certain class of manifolds in all dimensions (see \cite{Davaux}*{Theorem A\& B}).
Nevertheless, the hypothesis in \cite{Davaux} of non-vanishing $L^2$-index (see also \cite{MR3275037}) is a strong constraint, excluding important examples such as  (asymptotically) hyperbolic manifolds.

In this article, we aim to further investigate sharp bounds for the bottom spectrum of the Laplacian on complete Riemannian manifolds of higher dimensions by employing the Dirac operator and higher index theory, rather than relying on harmonic function methods. Higher index theory \cites{BaumConnesHigson,MR0996437,Roe} extends the classical notions of the Fredholm index and $L^2$-index, and it is particularly well-suited to the analysis of noncompact manifolds. In the case of universal covers of closed Riemannian manifolds, higher index theory is related to the Novikov conjecture of the fundamental group, as shown by J. Rosenberg (see \cite{MR720934}).
Here, we prove our first theorem on the sharp bottom spectrum and the scalar rigidity theorem for the universal cover.

\begin{theorem} \label{thm: cocompact_rigidity}
Suppose that $(M^n, g)$ is a closed Riemannian manifold with fundamental group $\Gamma$ and scalar curvature $\Sc_g \geq -\kappa$ for some constant $\kappa \geq 0$. If
\begin{itemize}
\item $M$ is rationally essential, namely the fundamental class $[M]$ is non-zero in $H_*(B\Gamma,\mathbb Q)$,
\item $\widetilde M$ is spin and $\Gamma$ satisfies the Strong Novikov Conjecture \ref{conj:Novikov}.
\end{itemize}
then
\begin{equation}\label{1.3}
    \lambda_1(\widetilde{M}, \widetilde{g}) \leq \frac{n-1}{4n} \kappa,
\end{equation}
where $(\widetilde{M}, \widetilde{g})$ denotes the Riemannian universal cover of $(M, g)$. Moreover, if the equality holds, then $\Sc_g \equiv -\kappa$.
\end{theorem}

In the case where $M$ is an aspherical manifold, the assumptions that $M$ is rationally essential and $\widetilde{M}$ is spin are automatically satisfied. Moreover, the Strong Novikov Conjecture (or its rational version, see \cite{MR4268302}) has no known counterexamples so far. Therefore, it is expected that line \eqref{1.3} holds for all aspherical manifolds. 

\begin{remark}
   If $\kappa = 0$ in Theorem \ref{thm: cocompact_rigidity}, then by using Ricci flow or solving Laplace equation (see \cites{Kazdan,GromovLawson}),  the work of Gromov–Lawson and Kazdan implies that $(M, g)$ is Ricci-flat ($\Ric_g \equiv 0$). In this paper, we provide a completely new argument using the Dirac operator method only together with the unique continuation theorem (see Proposition~\ref{prop: Ricci_flat}).
\end{remark}

Theorem \ref{thm: cocompact_rigidity} is equivalently stated as
\begin{equation} \label{eq: scalar_quantitative}
    \inf_{p \in M} \Sc_g(p) \leq -\frac{4n}{n-1} \lambda_1(\widetilde{M}, \widetilde{g}).
\end{equation}
Thus, Theorem \ref{thm: cocompact_rigidity} provides a quantitative obstruction to the existence of a complete Riemannian metric with nonnegative scalar curvature. More concretely, we may assume that $M$ is an aspherical manifold—such as in the two model cases: the torus and a closed hyperbolic manifold (see Section \ref{sec:Novikov} for more general cases). Note that Theorem \ref{thm: cocompact_rigidity} applies to both of these model cases as follows (see item (\ref{list: Hilbert}) and (\ref{list: hyperbolic}) in the list below Conjecture \ref{conj:Novikov}).

\begin{itemize}
  \item The torus $(\mathbb{T}^n, g)$, for which $\lambda_1(\widetilde{\mathbb{T}^n}, \widetilde{g}) = 0$ for any complete Riemannian metric $g$.
  \item A closed hyperbolic manifold $(M^n, g)$, for which $\lambda_1(\widetilde{M^n}, \widetilde{g}) > 0$ for any complete Riemannian metric $g$. This reflects the fact that the hyperbolicity of the fundamental group $\Gamma$ obstructs the increase of scalar curvature.
\end{itemize}

\bigskip
Moreover, Theorem \ref{thm: cocompact_rigidity} has the following geometric corollary (see item (\ref{list: npc}) in the list after Conjecture \ref{conj:Novikov}).

\begin{corollary}
    Suppose that $(M^n,g)$ is a closed Riemannian manifold with non-positive sectional curvature $\sec_g \leq 0$. If the scalar curvature $\Sc_g \geq -\kappa$ for some constant $\kappa \geq  0$, then
    \begin{equation}
    \lambda_1(\widetilde{M}, \widetilde{g}) \leq \frac{n-1}{4n} \kappa.
    \end{equation}
    Moreover, if the equality holds, then $\Sc_g \equiv -\kappa$.
\end{corollary}

Note that Xiaodong Wang proves that if $\Ric_g \geq -(n-1)$, then
\[
    \lambda_1(\widetilde{M}, \widetilde{g}) \leq \frac{(n-1)^2}{4}.
\]
In particular, equality holds if and only if $(\widetilde{M}, \widetilde{g})$ is isometric to the standard hyperbolic space $(\mathbb{H}^n, g_{\mathbb{H}^n})$ (see \cite{Wang_eigenvalue}*{Theorem~1.4}). Hence, Theorem~\ref{thm: cocompact_rigidity} naturally leads to the following geometric rigidity problem for $\kappa > 0$.

\begin{problem} \label{conj: rigidity}
	Under the same assumptions as in Theorem~\ref{thm: cocompact_rigidity}, if
\[
    \lambda_1(\widetilde{M}, \widetilde{g}) = \frac{(n-1)^2}{4},
\]
then the universal cover $(\widetilde{M}, \widetilde{g})$ is isometric to the space form with constant sectional curvature $\sec_g = -1$.
\end{problem}
Recall that Munteanu--Wang proved Problem~\ref{conj: rigidity} for closed three-dimensional aspherical manifolds using the harmonic function theory techniques (see \cite{MunteanuWang}*{Theorem~1.3}). However, the argument in \cites{MunteanuWang} cannot approach the scalar curvature rigidity part for higher dimensions in Theorem \ref{thm: cocompact_rigidity}. Hence, Problem \ref{conj: rigidity} remains open in full generality.

\vspace{2mm}

Next, we will study the general complete noncompact Riemannian manifold case. Let us recall that a complete Riemannian manifold $(X^n, g)$ is said to be \emph{geometrically contractible} if there exists a function $R(r) \geq r$ for any $r \geq 0$ such that $B(p,r)$ is contractible in $B(p, R(r))$ for any $p \in X$. Note that the universal Riemannian cover of any closed, aspherical Riemannian manifold is geometrically contractible (see \cite{wang2024fillingradiusquantitativektheory}*{Example 2.6}). A complete Riemannian manifold $(X^n, g)$ is said to be of bounded geometry if the sectional curvature and its derivatives are uniformly bounded, and the injective radius has a uniformly lower bound, i.e., $|\nabla^\alpha\sec_g| \leq K_\alpha, \ 
\inj(M) \geq i > 0$ for any multi-index $\alpha$ and constants $K_{\alpha ,i}>0$.

Now we are ready to state our second theorem on complete, noncompact Riemannian manifolds.

\begin{theorem} \label{thm: geometrically_contracbile_spectrum}
Suppose that $(X^n,g)$ is a complete, geometrically contractible Riemannian manifold with bounded geometry and scalar curvature $\Sc_g \geq -\kappa$ for some constant $\kappa\geq 0$. If $(X^n,g)$ satisfies the Coarse Novikov Conjecture \ref{conj:coarseNovikov}, then
	\begin{equation}
		\lambda_1(X,g) \leq \frac{n-1}{4n}\kappa.
	\end{equation}
	Moreover, if $\lambda_1(X,g) = \frac{n-1}{4n}\kappa$, then for any $\delta>0$, the set $$\{p \in X:\Sc_g(p)\geq -\kappa+\delta\}$$ is not a net of $(\title{X},g)$.
\end{theorem}

Recall that  a subset $S$ in $X$ is  said to be  a \emph{net} of $X$ if there exists $r>0$ such that $N_r(S)=X$, where $N_r(S) = \{x \in X: \dist(x, S) < r\}$. Theorem \ref{thm: geometrically_contracbile_spectrum} is a geometric version of a more general theorem presented in Section \ref{sec: proof_of_main_thm}. We emphasize that the upper bound $\frac{n-1}{4n}\kappa$ is sharp since the standard hyperbolic space $(\mathbb H^n,g_{\mathbb H^n})$ has scalar curvature $-n(n-1)$ and $\lambda_1(\mathbb H^n,g_{\mathbb H^n})=\frac{(n-1)^2}{4}$. Moreover, the net characterization in Theorem \ref{thm: geometrically_contracbile_spectrum} on the general complete, noncompact Riemannian manifold, as equality holds, is the best expectation and cannot be further improved in general.

\begin{example}\label{example}
	Let $X=\R^{n-1}\times\R$ be a  complete Riemannian manifold equipped with the metric
	$$g = dt^2 + \cosh^{\frac{2}{a}}(at)g_{\mathbb{R}^{n-1}},$$
	where $(n-1)/2\leq a<n/2$.
	Note that $(X, g)$ is a geometrically contractible manifold with bounded geometry and satisfies the coarse Novikov conjecture (see \cite[Chapter 7]{willett2020higher}).
	A direct calculation shows that
	\begin{enumerate}
		\item $\displaystyle\Sc_g = -n(n-1) + (n-1)\frac{n-2a}{\cosh^2(t)}> -n(n-1)$;
		\item $\lambda_1(X, g) = \frac{(n-1)^2}{4}$ (see \cite{Li_geometric_analysis}*{Proposition 22.2}) since 
		\begin{align*}
			&\Delta (\cosh^{- \frac{n-1}{2a}}(at))\\
			=& \cosh^{-\frac{n-1}{a}}(at)\partial_t( \cosh^{\frac{n-1}{a}}(at)\partial_t (\cosh^{-\frac{n-1}{2a}}(at))) \\
			=& \cosh^{-\frac{n-1}{a}}(at)\partial_t(-\frac{
				n-1}{2}\cosh^{\frac{n-1}{2a}-1}(at) \sinh(at)
			)\\
			=&-\frac{n-1}{2}\big(\frac{n-1}{2}-a\big)\cosh^{-\frac{n-1}{2a}-2}(at)\sinh^2(at)-a\frac{n-1}{2}\cosh^{-\frac{n-1}{2a}}(at)\\
			\leq& 
			-\frac{(n-1)^2}{4}\cosh^{- \frac{n-1}{2a}}(at).
		\end{align*}
        
	\end{enumerate}

\end{example}

\vspace{2mm}

\subsection*{Outline and Ideas of the Proofs}

In this paper, we recast the classical problem of the bottom spectrum upper bound estimate in geometric analysis as an instance of the classical Novikov conjecture within the framework of noncommutative geometry. We first prove a sharp upper bound on the bottom spectrum in terms of the scalar curvature lower bounds, and then further establish a corresponding geometric rigidity result. This paper is organized as follows:

In Section \ref{sec:pre}, we provide the necessary background on higher index theory and then prove that a nonzero index of the Dirac operator implies that zero lies in its spectrum (see Proposition \ref{prop: nozero_index_zero_spectrum}), which guarantees the existence of approximately harmonic spinors.

In Section \ref{sec: Kato}, we establish the refined Kato inequality. We present an elementary proof and a more general quantitative version for non-compact manifolds (see Proposition \ref{prop:kato} and Proposition \ref{prop:kato2}). This inequality plays a crucial role in establishing the sharpness of the bottom spectrum.

In Section \ref{sec: proof_of_main_thm}, we first prove a sharp upper bound on the bottom spectrum in both the noncompact and cocompact cases using the harmonic spinor as a test function, and then derive several corollaries. The main difficulty is to prove the scalar curvature rigidity theorem in the cocompact case, and to prove the net characterization in the general noncompact setting. Our approach relies heavily on a quantitative \emph{unique continuation theorem} (Theorem \ref{thm:uniqueCont}). For the reader’s convenience and potential future applications in geometric contexts, we postpone its proof to Section \ref{sec: unique_cont.}. To the best of our knowledge, this part is the first to apply the classical unique continuation theorem in the context of geometric (scalar curvature) rigidity characterization.


In Section \ref{sec: unique_cont.}, we provide the proof of the quantitative unique continuation theorem.

Our approach relies on higher index theory on non-compact manifolds, where $L^2$-harmonic spinors generally do not exist. Instead, one only has approximately harmonic spinors, in the sense that zero lies in the spectrum of the Dirac operator. To handle this situation, we develop new analytic techniques, including a quantitative refined Kato inequality (Proposition \ref{prop:kato2}) and a quantitative unique continuation theorem (Theorem \ref{thm:uniqueCont}).

\subsection*{Acknowledgement} We would like to thank Shiqi Liu, Yuguang Shi, Jiaping Wang, Xiaodong Wang, Xingyu Zhu, and Guoliang Yu for their interest and discussion on this topic. We would also like to thank the anonymous referees for their helpful comments.

\section{Preliminaries on higher index theory}\label{sec:pre}
In this section, we will review the construction of the geometric $C^*$-algebras and the higher index theory (see the textbook \cite{willett2020higher}). The higher index theory \cites{BaumConnesHigson,MR0996437,Roe} is a far-reaching generalization of the classical Fredholm index, particularly for non-compact manifolds, and is a more refined index theory than the Atiyah--Singer index theorem.

\textbf{Assumption}: All Riemannian manifolds considered in this paper in the context of index theory are assumed to be of bounded geometry.

\subsection{Roe algebras and localization algebras}\label{subsec:Roealgebra}
We will first review the definitions of some geometric $C^*$-algebras. 

Suppose that  $X$ is a proper metric space, i.e., every closed ball is compact. Let $\Gamma$ be a discrete group acting on $X$ by isometry. In the following, we only consider the cases where either $\Gamma$ is trivial, or $\Gamma$ acts properly and cocompactly.
Let $C_0(X)$ be the $C^*$-algebra consisting of all complex-valued continuous functions on $X$ that vanish at infinity. A $\Gamma$-$X$-module is a separable Hilbert space $H_X$ equipped with a $*$-representation $\varphi$ of $C_0(X)$ and an action $\pi$ of $\Gamma$, which are compatible in the sense that
$$\pi(\gamma)\big(\varphi(f)\xi\big)=\varphi(f^\gamma)\big(\pi(\gamma)\xi\big),~\forall f\in C_0(X),~\gamma\in\Gamma,~\xi\in H_X,$$
where $f^\gamma(x)\coloneqq f(\gamma^{-1}x)$.

A $\Gamma$-$X$-module $H_X$ is called \emph{admissible} if
\begin{enumerate}
	\item $H_X$ is \emph{nondegenerate}, namely the representation $\varphi$ is nondegenerate,
	\item $H_X$ is standard, namely no nonzero function in $C_0(X)$ acts as a compact operator, and
	\item for any $x\in X$, the stabilizer group $\Gamma_x$ acts on $H_X$ regularly, in the sense that the action is isomorphic to the action of $\Gamma_x$ on $l^2(\Gamma_x)\otimes H$ for some infinite-dimensional Hilbert space $H$.
\end{enumerate}

For example, if $X$ is a $\Gamma$-cover of a closed manifold, then $L^2(X)$ is naturally a $\Gamma$-$X$-module.

\begin{definition}\label{def:prop-locallycompact}
	Let $H_X$ be an admissible $\Gamma$-$X$-module and $T$ be a bounded linear operator acting on $H_X$.
	\begin{enumerate}
		\item The \emph{propagation} of $T$ is defined by
		$$
		\prop(T)=\sup\{d(x,y) ~|~ (x,y)\in \text{supp}(T) \},
		$$
		where $\text{supp}(T)$ is the complement (in $X\times X$) of the set of points $(x,y) \in X\times X$ such that there exists $f_1,f_2\in C_0(X)$ such that $f_1Tf_2=0$ and $f_1(x)f_2(y)\neq 0$;
		\item $T$ is said to be \emph{locally compact} if both $fT$ and $Tf$ are compact for all $f\in C_0(X)$.
		\item $T$ is said to be $\Gamma$-\emph{equivariant} if $\gamma T=T\gamma$ for any $\gamma\in\Gamma$.
	\end{enumerate}
\end{definition}

\begin{definition}\label{def roe and localization}
	Let $H_X$ be a standard nondegenerate $\Gamma$-$X$-module and $B(H_X)$ the set of all bounded linear operators on $H_X$.
	\begin{enumerate}
		\item The \emph{equivariant $Roe$ algebra} of $X$, denoted by $C^*(X)^\Gamma$, is the $C^*$-algebra generated by all locally compact, equivariant operators with finite propagation in $B(H_X)$.
		\item The \emph{equivariant localization algebra} $C_L^*(X)^\Gamma$ is the $C^*$-algebra generated by all bounded and uniformly norm-continuous functions $f\colon [1,\infty)\to C^*(X)^\Gamma$ such that 	
		\[
		\prop(f(t))<\infty \text{ and }\prop(f(t))\to 0\  \text{~as~}  t \to \infty.
		\]
	\end{enumerate}
\end{definition}
The Roe algebras and localization algebras of $X$ are independent (up to isomorphisms) of the choice of nondegenerate standard $\Gamma$-$X$-modules  $H_X$ (see \cite[Proposition 3.7]{Yulocalization}). 

There is a natural evaluation map
$$\ev\colon C^*_L(X)^\Gamma\to C^*(X)$$
induced by evaluating a path at $t=1$. The induced map $\ev_*$ at the level of $K$-theory is also usually referred to as the index map or the assembly map.

We will omit $\Gamma$ if $\Gamma$ is trivial. In the case where $\Gamma$ acts on $X$ properly and cocompactly, we have that $C^*(X)^\Gamma\cong C^*_r(\Gamma)\otimes \mathcal K$, where $C^*_r(\Gamma)$ is the reduced group $C^*$-algebra of $\Gamma$ and $\mathcal K$ is the algebra of compact operators. In particular, we have $K_*(C^*(X)^\Gamma)\cong K_*(C^*_r(\Gamma))$.
\subsection{Higher index and local higher index}\label{subsec:higher_index}
In this subsection, we will recall the definition of the higher index and local higher index for Dirac operators. 

Let $\chi$ be a continuous function on $\R$. $\chi$ is said to be a \emph{normalizing function} if it is non-decreasing, odd (i.e. $\chi(-x) = -\chi(x)$) and  
\[ \lim_{x\to \pm \infty} \chi(x) = \pm 1. \]  

Suppose that  $X$ is a complete spin manifold. Let $D$  be the associated Dirac operator on $X$ acting on the spinor bundle of $X$, and  $\Gamma$  is a discrete group acting on $X$  isometrically. Moreover, let $H$ be the Hilbert space of the $L^2$-sections of the spinor bundle, which is an admissible $\Gamma$-$X$-module in the sense of Section \ref{subsec:Roealgebra}.
Let us first assume that $\dim X$ is even. In this case, the spinor bundle is naturally $\Z_2$-graded, and the Dirac operator $D$ is an odd  operator given by
$$D=\begin{pmatrix}
	0&D_+\\D_-&0
\end{pmatrix}.$$

Let $\chi$ be a normalizing function. Since  $\chi$ is an odd function, we see that $\chi(t^{-1}D)$ is also a self-adjoint odd operator for any $t>0$ given by
\begin{equation}\label{eq:chi}
	\chi(t^{-1} D)=\begin{pmatrix}
		0&U_{t,D}\\V_{t,D}&0
	\end{pmatrix}.
\end{equation}
Now, we set
$$W_{t,D}=\begin{pmatrix}
	1& U_{t,D}\\0&1
\end{pmatrix}
\begin{pmatrix}
	1&0\\-V_{t,D}&1
\end{pmatrix}
\begin{pmatrix}
	1&U_{t,D}\\0&1
\end{pmatrix}
\begin{pmatrix}
	0&-1\\1&0
\end{pmatrix},\ e_{1,1}=\begin{pmatrix}
	1&0\\0&0
\end{pmatrix}
$$
and
\begin{equation}\label{eq:PtD}
	\begin{split}
		P_{t,D} = & W_{t,D}e_{1,1}W_{t,D}^{-1}\\
		= & \begin{pmatrix}
			1-(1-U_{t,D}V_{t,D})^2 & (2-U_{t,D}V_{t,D})U_{t,D}(1-V_{t,D}U_{t,D})\\
			V_{t,D}(1-U_{t,D}V_{t,D})&(1-V_{t,D}U_{t,D})^2
		\end{pmatrix}.
	\end{split}
\end{equation}
The path $(P_{t,D})_{t\in[1,+\infty)}$ defines an element in  $M_2((C^*_L(X)^\Gamma)^+)$, and the difference $P_{t,D}-e_{1,1}$ lies in $M_2(C^*_L(X)^\Gamma)$.
\begin{definition}\label{def:lind}
	If $X$ is a  spin manifold of even dimension, then  
	\begin{itemize}
		\item 	the local higher index $\ind_L(D)$ of $D$ is defined to be  
		\[ \ind_L(D)\coloneqq [P_{t,D}]-[e_{1,1}] \in K_0(C^*_L(X)^\Gamma); \] 
		\item 
		the higher index $\ind(D)$ of $D$ is defined to be  
		\[ \ind(D) \coloneqq [P_{1,D}]-[e_{1,1}] \in  K_0(C^*(X)^\Gamma).\]
	\end{itemize}

\end{definition}

The constructions of the (local) higher index for the odd-dimensional spin manifold are as follows.
\begin{definition}\label{def:ind-odd}
If $X$ is a spin manifold of odd dimension, then
\begin{itemize}
	\item  the local higher index  $\ind_L(D)$ of $D$  is defined to be 
	\[ [e^{2\pi i\frac{\chi(t^{-1}D)+1}{2}}] \in K_1(C^*_L(X)^\Gamma);\]
	
	\item  the higher index  $\ind(D)$ of $D$ is defined to be  
	\[ [e^{2\pi i\frac{\chi(D)+1}{2}}] \in K_1(C^*(X)^\Gamma). \]
\end{itemize}

\end{definition}

Note that the higher index and the local higher index are independent of the choices of normalizing functions. The $K$-theory $K_\ast(C^*_L(X)^\Gamma)$ of the localization algebra $C^*_L(X)^\Gamma$ is naturally isomorphic to  the $\Gamma$-equivariant $K$-homology of $X$. Under this isomorphism, the local higher index of $D$ coincides with the $K$-homology class of $D$ (see \cite[Theorem 3.2]{Yulocalization} and \cite[Theorem 3.4]{QiaoRoe}).

\begin{proposition} \label{prop: nozero_index_zero_spectrum}
	Suppose that $(X,g)$ is a spin Riemannian manifold and $D$ is the associated Dirac operator acting on the spinor bundle.	If $\ind(D)\ne 0$ in $K_*(C^*(X)^\Gamma)$, then zero is in the spectrum of $D$.
\end{proposition}
\begin{proof}
Assume that $D$ is an invertible operator on a spin manifold $X$, namely $0$ is not in the spectrum of $D$. Then we choose the normalizing function $\chi$ to be the following function
$$\chi(x)=\begin{cases}
	1&x\geq 0,\\-1&x<0,
\end{cases}$$
which is continuous on the spectrum of $D$, and satisfies $\chi(D)^2=1$. Consequently, we reach that
\begin{itemize}
	\item if $X$ has even dimension, then $P_{1,D}=e_{1,1}$;
	\item if $X$ has odd dimension, then $$e^{2\pi i \frac{\chi(D)+1}{2}}=1.$$
\end{itemize}
 It follows that $\ind(D)=0$. 
\end{proof}


\subsection{Strong Novikov Conjecture and its coarse analogue}\label{sec:Novikov}
In this subsection, we recall the statement of the Strong Novikov Conjecture for groups and its coarse analogue for non-compact metric spaces.

Let $(X,d)$ be a discrete metric space with bounded geometry.
For each $d>0$, we define the Rips complex $P_d(X)$ to be the simplicial complex generated by points in $X$ such that $x_i,x_j\in X$ are in the same simplex if $d(x_i,x_j)\leq d$. By construction, $P_d(X)$ is finite-dimensional. We equip $P_d(X)$ with the spherical metric: for each simplex
$$\{\sum_{k=1}^m t_kx_{i_k}:\sum_{k=1}^n t_k=1,~t_k\geq 0\}.$$
Its metric is the one obtained from the sphere $\sph^m$ through the following map:
$$\sum_{k=1}^m t_kx_{i_k}\mapsto\Big(\frac{t_0}{\sqrt{\sum_{k=1}^n t_k^2}},\cdots,\frac{t_0}{\sqrt{\sum_{k=1}^n t_k^2}}\Big).$$
In particular, if $X=\Gamma$ is a finitely presented group, then $P_d(\Gamma)$ admits a natural $\Gamma$-action, which is proper and cocompact. We similarly define its Roe algebra and localization algebra. 
In particular, the Roe algebras (or the equivariant version) of $X$ and $P_d(X)$ are isomorphic (see \cite[Proposition 3.7]{Yulocalization}).
\begin{conjecture}[Coarse Novikov Conjecture]\label{conj:coarseNovikov}
	Let $X$ be a discrete metric space with bounded geometry. The coarse Novikov conjecture for $X$ states that the evaluation map 
	$$\ev\colon \lim_{d\to\infty} C^*_L(P_d(X))\to  C^*(X)$$
	induces an injection
	$$\ev_*\colon \lim_{d\to\infty}K_*( C^*_L(P_d(X)))\to K_*( C^*(X)).$$
\end{conjecture}
\begin{conjecture}[Strong Novikov Conjecture]\label{conj:Novikov}
	Let $\Gamma$ be a finitely presented group. The Strong Novikov Conjecture for $\Gamma$ states that the evaluation map 
	$$\ev\colon \lim_{d\to\infty} C^*_L(P_d(\Gamma)^\Gamma)\to  C^*_r(\Gamma)$$
	induces an injection 
	$$\ev_*\colon \lim_{d\to\infty}K_*( C^*_L(P_d(\Gamma))^\Gamma)\to K_*( C^*_r(\Gamma)).$$
\end{conjecture}
We briefly recall some known cases for the Coarse (Strong) Novikov Conjectures. 

The Strong Novikov Conjecture (or its rational version) holds for groups as follows:
\begin{enumerate}
	\item \label{list: npc} groups acting properly and isometrically on simply connected and non-positively curved manifolds(see \cite[Section 5.3]{GK88}).
	\item \label{list: hyperbolic}hyperbolic groups (see \cite[Theorem 6.8]{ConnesMoscovici}).
	\item \label{list: Hilbert}groups acting properly and isometrically on Hilbert spaces (see \cite[Theorem 1.1]{MR1821144}), for example, amenable groups.
	\item groups acting properly and isometrically on bolic spaces (see \cite[Theorem 1.1]{MR1998480}).
	\item \label{list: FSD}groups with finite asymptotic dimension (see \cite[Corollary 7.2]{Yu}),
	\item groups that coarsely embed into Hilbert spaces (see\cite[Corollary 1.2]{Yucoarseembed}).
    \item all subgroups of almost connected Lie groups (see \cite{MR2217050}).
    \item groups acting properly and isometrically on an admissible Hilbert-Hadamard space (see \cite[Theorem 1.1]{MR4268302}).
\end{enumerate}

The Coarse Novikov Conjecture holds for metric spaces as follows.
\begin{enumerate}
	\item metric spaces that are coarsely equivalent to non-positively curved manifolds (see \cite[Section 4]{Yulocalization}).
	\item metric spaces that have finite asymptotic dimension (see \cite[Theorem 1.1]{Yu}).
    \item metric spaces that have subexponential volume growth (see \cite[Theorem 5.1]{MR1791141}).
	\item metric spaces that coarsely embed into Hilbert spaces (see \cite[Theorem 1.1]{Yucoarseembed}).
    \item metric spaces that coarsely embed into Hilbert spaces (see \cite[Theorem 1.1]{MR3325537}).
\end{enumerate}

In particular, we remark that by the descent principle \cite[Theorem 8.4]{Roe}, the isomorphism of the map $\ev_*$ in Conjecture \ref{conj:coarseNovikov} for a group $\Gamma$ (as a metric space) implies the Strong Novikov Conjecture of the group $\Gamma$.

\section{A quantitative refined Kato inequality} \label{sec: Kato}

The Kato inequality for harmonic spinors is essential for us to obtain the sharpness of the bottom spectrum. In this subsection, we will provide a detailed proof of the Kato inequality, inspired by \cite[Section 4.1]{Davaux} and \cite{CGH00}, in order to present a slightly more general version (see Proposition \ref{prop:kato2}) for noncompact manifolds. For the simplicity of the notation, we only provide the proof for real spinors, while the complex case also holds with the same argument.

\begin{proposition}\label{prop:kato}
	Suppose that $(X^n,g)$ is a complete Riemannian manifold and $E$ is a vector bundle over $X$ equipped with a Clifford action of $TX$. Let $D$ be the Dirac operator
	$$D=\sum_{i=1}^n c(e_i)\nabla_{e_i},$$
	where $\nabla$ is a connection on $E$.
	If $\xi$ is a smooth section of $E$ such that $D\xi=0$, then
	$$\Big|\nabla|\xi|\Big|^2\leq \frac{n-1}{n}|\nabla\xi|^2.$$
\end{proposition}
\begin{proof}
	Let $\nabla\xi$ be the derivative of $\xi$ as a section in $TX\otimes E$.
	Note that
	$$\Big|\nabla|\xi|^2\Big|=2\Big|\nabla|\xi|\Big|\cdot|\xi|=2\Big|\langle\nabla\xi,\xi\rangle\Big|.$$
	Therefore, if $\xi(x)\ne 0$ for $x\in X$, then the desired inequality at $x$ is equivalent to
	$$\Big|\langle\nabla\xi(x),\xi(x)\rangle\Big|^2\leq \frac{n-1}{n}|\nabla\xi(x)|^2|\xi(x)|^2.$$
	Since $D\xi(x)=0$, we have $\nabla\xi(x)\in\ker T$, where $T$ is the endomorphism
	$$T\colon (TX\otimes E)_x\to E_x,~\psi\mapsto \sum_{i=1}^n c(e_i)\langle\psi,e_i\rangle.$$
	The inequality now follows from Lemma \ref{lemma:kato}, which will be proved later. Therefore, we have shown that
	$$\Big|\big(\nabla|\xi|\big)(x)\Big|^2\leq \frac{n-1}{n}|\big(\nabla\xi\big)(x)|^2$$
	for any $x\in\supp(\xi)$, namely the support of $\{x\in X:\xi(x)\ne 0\}$. The inequality holds trivially outside $\supp(\xi)$. This finishes the proof.
\end{proof}
\begin{lemma}\label{lemma:kato}
	Suppose that $V$ is a vector space and  $W$ is a vector space equipped with a $Cl(V)$-action. Let
	$$T\colon V\otimes W\to W,~\psi\mapsto \sum_{i=1}^n c(e_i)\langle\psi,e_i\rangle,$$
	then, for any $\psi\in \ker T$ and $\xi\in V$, we have
	$$\Big|\langle\psi,\xi\rangle\Big|^2\leq \frac{n-1}{n}|\psi|^2|\xi|^2.$$
\end{lemma}
\begin{proof}
	Let $\psi=\sum_{i=1}^n e_i\otimes s_i$. Since $\psi\in \ker T$, we have
	$$\sum_{i=1}^n c(e_i)s_i=0.$$
	Now it suffices to prove that
	\begin{equation}\label{eq:sixi}
		\sum_{i=1}^n\langle s_i,\xi\rangle^2\leq \frac{n-1}{n}\sum_{i=1}^n|s_i|^2|\xi|^2
	\end{equation}
	subject to the equality for $s_i$'s above.
	
	Assume that $|\xi|=1$. We will prove by induction on $n$. The case when $n=1$ is obvious, as $Ts=0$ implies that $s=0$. When $n=2$, we have
	$$c(e_1)s_1+c(e_2)s_2=0,$$
	namely $s_1=\omega s_2$, where $\omega=c(e_1)c(e_2)$. Observe that $\omega^*=-\omega$ and $\omega^2=-1$. Hence 
	\begin{equation}\label{eq:perp}
		|s_2|=|\omega s_2|\text{ and }s_2\perp\omega s_2.
	\end{equation} It follows that
	$$\sum_{i=1}^n\langle s_i,\xi\rangle^2=\langle\omega s_2,\xi\rangle^2+\langle s_2,\xi\rangle^2=|s_2|^2\Big|P(\xi)\Big|^2\leq \frac 1 2\Big(|s_1|^2+|s_2|^2\Big)|\xi|^2,$$
	where $P$ is the orthogonal projection from $W$ to $\text{span}\{s_2,\omega s_2\}$.
	This finishes the proof when $n=2$. In particular, the equality holds if and only if $\text{span}\{s_2,\omega s_2\}$ or $s_2=0$.
	
	Now we prove the inequality \eqref{eq:sixi} for $n\geq 3$ by induction. For any $i\ne j$, we define
	$$s_{i,j}=s_i-\frac{1}{n-1}c(e_i)c(e_j)s_j.$$
	Since $\sum_{i=1}^n c(e_i)s_i=0$, we see that $\sum_{i:i\ne j} c(e_i)s_{i,j}=0.$ By the induction hypothesis, we have
	$$\sum_{i:i\ne j}\langle s_{i,j},\xi\rangle^2\leq \frac{n-2}{n-1}\sum_{i:i\ne j}|s_{i,j}|^2.$$
	Take summation for $j=1,2,\ldots,n$ and obtain that
	\begin{equation}\label{eq:ij}
		\sum_{i\ne j}\langle s_{i,j},\xi\rangle^2\leq \frac{n-2}{n-1}\sum_{i\ne j}|s_{i,j}|^2
	\end{equation}
	
	We first compute the right-hand side of line \eqref{eq:ij}. Note that
	\begin{align*}
		|s_{i,j}|^2=&\Big|s_i-\frac{1}{n-1}c(e_i)c(e_j)s_j\Big|^2=|s_i|^2+\frac{1}{(n-1)^2}|s_j|^2-\frac{2}{n-1}\langle s_i,c(e_i)c(e_j)s_j\rangle\\
		=&|s_i|^2+\frac{1}{(n-1)^2}|s_j|^2+\frac{2}{n-1}\langle c(e_i)s_i,c(e_j)s_j\rangle
	\end{align*}
	Sum for $i$ with $i\ne j$,
	$$\sum_{i:i\ne j}|s_{i,j}|^2=\sum_{i:i\ne j}|s_i|^2+\frac{1}{n-1}|s_j|^2-\frac{2}{n-1}|s_j|^2=\sum_{i:i\ne j}|s_i|^2-\frac{1}{n-1}|s_j|^2.$$
	It follows that
	\begin{equation}\label{eq:rhs-ij}
		\begin{split}
			\sum_{i\ne j}|s_{i,j}|^2=&\sum_{i\ne j}|s_i|^2-\frac{1}{n-1}\sum_{j=1}^n|s_j|^2=(n-1)\sum_{i=1}^n|s_i|^2-\frac{1}{n-1}\sum_{j=1}^n|s_j|^2\\
			=&\frac{n(n-2)}{n-1}\sum_{i=1}^n|s_i|^2.
		\end{split}
	\end{equation}
	
	Now we estimate the left-hand side of line \eqref{eq:ij}. Fix $i\in\{1,\ldots,n\}$. We have
	$$\sum_{j:i\ne j}\langle s_{i,j},\xi\rangle^2=\sum_{j:i\ne j}\langle s_i-\frac{1}{n-1}c(e_i)c(e_j)s_j,\xi\rangle^2$$ 
	By the Cauchy--Schwarz inequality, we have
	\begin{equation}
		\begin{split}
			&(n-1)\sum_{j:i\ne j}\langle s_i-\frac{1}{n-1}c(e_i)c(e_j)s_j,\xi\rangle^2\\
			\geq&\Big(\sum_{j:i\ne j}\langle s_i-\frac{1}{n-1}c(e_i)c(e_j)s_j,\xi\rangle\Big)^2\\
			=&\Big\langle (n-1)s_i-\frac{1}{n-1}\sum_{j:i\ne j}c(e_i)c(e_j)s_j,\xi\Big\rangle^2\\
			=&\frac{n^2(n-2)^2}{(n-1)^2}\langle s_i,\xi\rangle^2.
		\end{split}
	\end{equation}
	Here, the last equality follows from
	$$\sum_{j:i\ne j}c(e_i)c(e_j)s_j=c(e_i)\Big(-c(e_i)s_i\Big)s_i=s_i.$$
	Thus, we obtain an estimate for the left-hand side of line \eqref{eq:ij} as
	\begin{equation}\label{eq:lhs-ij}
		\sum_{i\ne j}\langle s_{i,j},\xi\rangle^2\geq \frac{n^2(n-2)^2}{(n-1)^3}\sum_{i=1}^n\langle s_i,\xi\rangle^2.
	\end{equation}
	Combining \eqref{eq:ij}, \eqref{eq:rhs-ij}, and \eqref{eq:lhs-ij}, we obtain that
	$$\frac{n^2(n-2)^2}{(n-1)^3}\sum_{i=1}^n\langle s_i,\xi\rangle^2\leq \frac{n-2}{n-1}\cdot \frac{n(n-2)}{n-1}\sum_{i=1}^n|s_i|^2.$$
	Since $n\geq 3$, we have
	$$\sum_{i=1}^n\langle s_i,\xi\rangle^2\leq \frac{n-1}{n}\sum_{i=1}^n|s_i|^2.$$
	This finishes the proof.
\end{proof}

Indeed, the proof of Lemma \ref{lemma:kato} implies a quantitative version of Proposition \ref{prop:kato} as follows.

\begin{proposition}\label{prop:kato2}
	If $E$ is a bundle over $X^n$ equipped with a Clifford action of $TX$, then there exists $c_n>0$  depending only on $n$ such that, for any smooth section $\xi$ of $E$, we have
	$$\Big|\nabla|\xi|\Big|^2\leq \frac{n-1}{n}|\nabla\xi|^2+c_n|D\xi|^2+c_n|D\xi||\nabla\xi|.$$
\end{proposition}

\section{Sharp Bottom spectrum and scalar curvature rigidity} \label{sec: proof_of_main_thm}
In this section, we will prove the main theorems and then state several related corollaries. 
\subsection{Complete manifolds with cocompact action}
We first prove the result for universal covers of closed manifolds. Let us first state a technical ingredient, the proof of which will be postponed to Section \ref{sec: unique_cont.}.
\begin{proposition}[See Theorem \ref{thm:uniqueCont} in Section \ref{sec: unique_cont.}]\label{prop:uniqueCont}
Suppose that $(X^n, g)$ is a complete spin Riemannian manifold with bounded geometry and  $Y$ a subset of $X$. Let $N_a(Y)$ be the $a$-neighborhood of $Y$ for some $a>0$. Let $D$ be the Dirac operator of $X$. Let  $P_\lambda$ be the spectral projection of $D^2$ with spectrum $\leq \lambda$ and $V_\lambda$ the range of $P_\lambda$. If there exists $r>0$ such that $N_r(Y)=X$, then there exists a constant $C_\lambda>0$ such that
	$$\|\sigma\|_{L^2(X)}\leq C_\lambda\|\sigma\|_{L^2(N_a(Y))}\text{ for any }\sigma\in V_\lambda,$$
    where $C_\lambda\leq c_1e^{c_2\lambda}$ for some $c_1,c_2>0$.
\end{proposition}

\begin{theorem}\label{thm:cover}
	Suppose that $(M^n,g)$ is a closed Riemannian manifold and $(\widetilde M,\widetilde g)$ is the Riemannian universal cover of (M, g). Assume that $\widetilde M$ is spin and $\widetilde D$ is the Dirac operator acting on the spinor bundle over $\widetilde M$. If
	\begin{enumerate}
		\item $\ind(\widetilde D)\in K_*(C^*(\widetilde M)^\Gamma)\cong K_*(C^*_r(\Gamma))$ is non-zero;
		\item $\Sc_{g}\geq -\kappa$ for some constant $\kappa\geq 0$,
	\end{enumerate} 
	then $$\lambda_1(\widetilde M,\widetilde g)\leq \frac{n-1}{4n}\kappa.$$
	Moreover, if $\lambda_1(\widetilde M,\widetilde g)= \frac{n-1}{4n}\kappa$,  then $(M^n, g)$ has constant scalar curvature $\Sc_{g}=-\kappa$ on $M$.
\end{theorem}

\begin{proof}
   	Let $S_{\widetilde M}$ be the spinor bundle over $\widetilde M$. Since $\ind(\widetilde D)\in K_*(C^*(\widetilde M)^\Gamma)$ is non-zero, we obtain that the Dirac operator $D$ is not invertible by Proposition \ref{prop: nozero_index_zero_spectrum}.  Consequently, for any $\varepsilon>0$, there exists a spinor $s\in L^2(S_{\widetilde M})$ such that
	$$\|s\|=1 \text{ and }\|\widetilde Ds\|\leq\varepsilon.$$
    Recall that
	\begin{itemize}
		\item  The Lichnerowicz formula \[\Delta = \nabla^*\nabla + \frac{1}{4}\Sc_{\widetilde g}.\]

        shows that 
		$$\|\nabla s\|^2=\|D s\|^2-\int_{\widetilde M}\frac{\Sc_{\widetilde g}}{4}|s|^2\leq \varepsilon^2+\frac{\kappa}{4}.$$
		
		\item 	The  Kato inequality in Proposition \ref{prop:kato2} indicates that there exists $c_n>0$ such that
		$$\Big|\nabla|s|\Big|^2\leq \frac{n-1}{n}|\nabla s|^2+c_n|D s|^2+c_n|Ds||\nabla s|$$
		in $(\widetilde M, \widetilde {g})$.
	\end{itemize}
    \bigskip
	By Integrating on $\widetilde M$, we obtain that
	\begin{align*}
		\int_{\widetilde M}\langle-\Delta |s|,|s|\rangle=\Big\|\nabla|s|\Big\|^2
		\leq \frac{n-1}{n}\big(\varepsilon^2+\frac{\kappa}{4}\big)+c_n\varepsilon^2+c_n\varepsilon\sqrt{\varepsilon^2+\frac{\kappa}{4}}.
	\end{align*}
 Since $\varepsilon > 0$ can be chosen arbitrarily, we take the limit $\varepsilon \to 0^+$ and conclude that $$\lambda_1(\widetilde M,\widetilde g)\leq \frac{n-1}{4n}\kappa.$$
	
Next, let us prove the scalar curvature rigidity if the equality holds. Assume otherwise that for some $\delta>0$, the open set 
    $$U=\{x\in M:\Sc_g(x)>-\kappa+\delta\}$$
    is non-empty. 
    
    Let $\widetilde U$ be the lift of $U$ in $\widetilde M$ and we define the $r$-neighborhood (denoted by $\widetilde U_r$) of $\widetilde U$ as
    \[\widetilde U_r = \{x \in \widetilde M, \ \dist_{\widetilde g}(x, \widetilde U) \leq r\}.\]
    Then there exists a constant $r> 0$ such that $\widetilde U_r = \widetilde M$ with $r$ at most the diameter of $M$. Moreover, for  any given $\varepsilon>0$, let $P_{\varepsilon^2}$ be the spectral projection to the spectrum $\leq\varepsilon^2$ and $V_{\varepsilon^2}$ the range of $P_{\varepsilon^2}$. Since $\widetilde D$ is non-invertible,  we obtain that $V_{\varepsilon^2}$ is non-empty. Let us pick a spinor $s$ in $V_{\varepsilon^2}$ with $\|s\|=1$. Clearly we have $\|\widetilde Ds\|\leq\varepsilon.$  By our assumption, we obtain that  $$\Sc_{\widetilde g}\geq -\kappa+\delta  \text{ on } \widetilde U.$$ Hence, by the Lichnerowicz formula, we get
	$$\|\nabla s\|^2=\|D s\|^2-\int_{\widetilde M}\frac{\Sc_{\widetilde g}}{4}|s|^2\leq \varepsilon^2+\frac{\kappa}{4}-\frac{\delta}{4}\|s\|^2_{L^2(\widetilde U)}.$$
	Similarly, we deduce
	$$\Big\|\nabla|s|\Big\|^2
	\leq \frac{n-1}{n}\big(\varepsilon^2+\frac{\kappa}{4}\big)+c_n\varepsilon^2+c_n\varepsilon\sqrt{\varepsilon^2+\frac{\kappa}{4}}-\frac{(n-1)\delta}{4n}\|s\|^2_{L^2(\widetilde U)}.$$
	Assume that $\varepsilon<1$. By Proposition \ref{prop:uniqueCont}, there exists $C>0$ independent of $\varepsilon$ such that
	$$\|s\|_{L^2(\widetilde U)}\geq \frac 1 C\|s\|=\frac 1 C.$$
	Therefore, we see that
	$$\Big\|\nabla|s|\Big\|^2
	\leq\frac{(n-1)\kappa}{4n}-\frac{(n-1)\delta}{4nC^2}+\Big(	
	\frac{n-1}{n}\varepsilon^2+c_n\varepsilon^2+c_n\varepsilon\sqrt{\varepsilon^2+\frac{\kappa}{4}}\Big).$$
	By letting $\varepsilon\to 0$, we reach
	$$\lambda_1(\widetilde M,\widetilde g)\leq  \frac{n-1}{4n}\kappa-\frac{(n-1)\delta}{4nC^2}<\frac{n-1}{4n}\kappa.$$
	This contradicts the assumption that $\lambda_1(X,g)=\frac{n-1}{4n}\kappa$ and finishes the proof.
\end{proof}

Next, we remark that the index-theoretic condition, namely $\ind(D) \in K_*(C^*(X)^\Gamma)$ being nonzero, can be verified under the following topological and algebraic conditions.

\begin{proposition}
	Suppose that $(M^n, g)$ is a closed Riemannian manifold and $\Gamma=\pi_1(M)$. If
	\begin{itemize}
		\item $M$ is rationally essential, namely the fundamental class $[M]$ is non-zero in $H_*(B\Gamma,\mathbb Q)$
        \item $\widetilde M$ is spin, and
		\item $\Gamma$ satisfies the Strong Novikov Conjecture \ref{conj:Novikov} (or its rational version),
	\end{itemize}
	then $\ind(\widetilde D)$ is in non-zero in $K_*(C^*_r(\Gamma))$. Hence, Theorem \ref{thm: cocompact_rigidity} holds.
\end{proposition}
\begin{proof}
    Assume first $M$ itself is spin, and $D$ the Dirac operator on $M$. Let $B\Gamma$ be the classifying space of $\Gamma$, $E\Gamma$ its universal cover, and $h\colon M\to B\Gamma$ the classifying map. As equivariant homology classes are, by definition, locally finitely supported, we have a natural map
    $$H^\Gamma_*(B\Gamma,\mathbb Q)=H_*^\Gamma(E\Gamma,\mathbb Q)\to \lim_{d\to\infty} H_*^\Gamma(P_d(\Gamma),\mathbb Q),$$
    which is injective by \cite[(7.4)]{BaumConnesHigson}. A Mayer-Vietoris argument shows that
    $$H^*(P_d(\Gamma),\mathbb Q)\cong K_*(C^*_L(P_d(\Gamma))\otimes\mathbb Q$$
    via the Chern character map. Therefore, if otherwise the higher index of $\widetilde D$ vanishes in $K_*(C^*_r(\Gamma))$, then the Chern character of $D$ vanishes in $H_*(B\Gamma,\mathbb Q)$. As a result, for any class $\alpha\in H^*(B\Gamma,\mathbb Q)$, we have 
    $$\langle\hat A(M)\cup h^*(\alpha),[M]\rangle=0,$$
    where $\hat A(M)$ is the $\hat A$-class of $M$ given by
    $$\hat A(M)=1-\frac{1}{24} p_1(M)+\cdots\in H^*(M,\mathbb Q).$$
    In particular, we have $\langle h^*(\alpha),[M]\rangle=0$ for any $\alpha\in H^n(B\Gamma,\mathbb Q)$, which contradicts to that $M$ is rationally essential.

    In general, the proof follows from the same argument as in the proof of \cite[Theorem 3.5]{MR720934} by considering instead the $\hat \Gamma$-equivariant index with $\hat \Gamma$ constructed from the pull-back diagram
    $$\begin{tikzcd}
\hat\Gamma \arrow[r] \arrow[d] & Aut(P_{Spin}) \arrow[d] \\
\Gamma \arrow[r]               & Aut(P_{SO})            
\end{tikzcd},$$
where $P_{SO}$ is the principal $SO(n)$-bundle over $\widetilde M$ and $P_{Spin}$ is the principal $Spin(n)$-bundle over $\widetilde M$.
\end{proof}

Note that our scalar curvature rigidity result follows from a spectral argument. This idea was previously used in \cites{WangXie25,MR4855340}. In fact, we can further apply this type of argument to the special case of the rigidity part where $\kappa=0$, from which we can deduce that $M$ is Ricci flat (see \cite{GromovLawson}). Classically, the fact that $M$ is Ricci flat follows from a result of Kazdan (see \cite{Kazdan}). Here, we provide a new proof that depends only on the technique of Dirac operators and \emph{unique continuation theorem} in this paper.

\begin{proposition} \label{prop: Ricci_flat}
    With the same notation and assumptions as in Theorem \ref{thm:cover}, if $\kappa = 0$, then $\Ric_g \equiv 0$ on $M$. 
\end{proposition}
\begin{proof}
    Note that Theorem \ref{thm:cover} implies  $\Sc_g=0$  on $M$.
    Now we assume  that $Ric_g$ is not identically zero on $M$, then it means that, for some $\delta>0$, the open set
    $$U=\{x\in M: |Ric_g(x)|>\delta\}$$
    is non-empty. Let $\widetilde U$ be the lift of $U$ in $\widetilde M$. 

    Since by assumption that $\widetilde D$ is non-invertible, for any $\varepsilon>0$, there exists a spinor $s$ of $\widetilde M$ such that
    $$\|s\|=1\text{ and }\|\widetilde D s\|\leq\varepsilon.$$
    The Lichnerowicz formula shows that
    $$\|\nabla s\|^2=\|Ds\|^2\leq \varepsilon^2.$$
    Let $c$ be the Clifford action of $S_{\widetilde M}$ and assume the local orthonormal basis $\{e_i\}$ of $\widetilde M$, then we obtain,  by \cite[Corollary 2.9]{spinorialapproach}
    $$\sum_{j=1}^nc(e_i)(\nabla_{e_i}\nabla_{e_j}-\nabla_{e_j}\nabla_{e_i}-\nabla_{[e_i,e_j]})s=-\frac 1 2c(Ric_{\widetilde g}(e_i))s.$$
    Hence
    $$\langle\sum_{j=1}^nc(e_i)(\nabla_{e_i}\nabla_{e_j}-\nabla_{e_j}\nabla_{e_i}-\nabla_{[e_i,e_j]})s,-\frac 1 2c(Ric_{\widetilde g}(e_i))s\rangle=\frac 1 4|Ric_{\widetilde g}(e_i)|^2\cdot|s|^2.$$
    Therefore, by the Stokes formula and a partition of unity, there exists $c>0$ that only depends on $n$ and $g$ such that
    $$\int_{\widetilde M}|Ric_g|^2|s|^2\leq c\int_{\widetilde M}|\nabla s|^2.$$
    Thus we obtain
    $$\delta^2\int_{\widetilde U}|s|^2\leq \varepsilon^2.$$
    Recall that Poincar\'{e} inequality on $\widetilde M$ (see \cite{WangXie25}*{Lemma 2.7}) as follows.
    $$\|s\|^2\leq C_1\|\nabla s\|^2+C_2\int_{\widetilde U}|s|^2$$
    for some $C_1,C_2>0$. Hence, we reach that $$\|s\|^2\leq (C_1+C_2/\delta^2)\varepsilon^2$$ which leads to a contradiction that $\|s\|=1$. Therefore, we conclude that $\Ric(g) = 0$ on $M$.
\end{proof}

\begin{remark}
 Under the additional assumption that $(M, g)$ is aspherical in Theorem~\ref{thm: cocompact_rigidity}, the condition $\kappa = 0$ implies that $(M, g)$ is flat. However, we note that when $\kappa = 1$, the corresponding geometric rigidity result asserting that $\sec = -1$ is significantly more subtle and remains nontrivial (see Problem \ref{conj: rigidity}).
\end{remark}

\subsection{Complete manifolds without group action}
In this subsection, we prove the main result for complete manifolds. Let us first introduce some basic concepts.

\begin{definition}\label{def:net}
	Suppose that  $(X, d)$ is a metric space and $S$ is a subset of $X$. $S$ is said to be  a \emph{net} of $X$ if there exists $r>0$ such that $N_r(S)=X$, where $N_r(S) = \{x \in X: \dist(x, S) < r\}$. Furthermore, we say that $S$ is a \emph{discrete net} of $X$ if there exists $r'>0$ such that $d(x,y)\geq r'$ for any $x\ne y$ in $S$.
\end{definition}

\begin{theorem}\label{thm:complete}
	Suppose that $(X^n, g)$ is a complete, noncompact, spin Riemannian manifold with bounded geometry and  $D$ is the Dirac operator acting on the spinor bundle over $X$. If
	\begin{enumerate}
		\item $\ind(D)\in K_*(C^*(X))$ is non-zero, and
		\item $\Sc_{g}\geq -\kappa$ for some constant $\kappa\geq 0$,
	\end{enumerate} 
	then $$\lambda_1(X,g)\leq \frac{n-1}{4n}\kappa.$$
Moreover, if $\lambda_1(X,g)= \frac{n-1}{4n}\kappa$, then for any $\delta>0$, the set $$\{x\in X:\Sc_g(x)\geq -\kappa+\delta\}$$ is not a net of $X$.
\end{theorem}

\begin{proof}
   	Let $S_{X}$ be the spinor bundle over $X$. Since $\ind(D)\in K_*(C^*(X))$ is non-zero, we obtain that the Dirac operator $D$ is not invertible.  Consequently, for any $\varepsilon>0$, there exists a spinor $s\in L^2(S_{X})$ such that
	$$\|s\|=1 \text{ and }\|Ds\|\leq\varepsilon.$$
	Note that
	\begin{itemize}
		\item  The Lichnerowicz formula shows that 
		$$\|\nabla s\|^2=\|D s\|^2-\int_{X}\frac{\Sc_{g}}{4}|s|^2\leq \varepsilon^2+\frac{\kappa}{4}.$$
		
		\item 	The  Kato inequality in Proposition \ref{prop:kato2} indicates that there exists $c_n>0$ such that
		$$\Big|\nabla|s|\Big|^2\leq \frac{n-1}{n}|\nabla s|^2+c_n|D s|^2+c_n|Ds||\nabla s|$$
		in $(X, {g})$.
	\end{itemize}
	By Integrating on $X$, we obtain that
	\begin{align*}
		\int_{X}\langle-\Delta |s|,|s|\rangle=\Big\|\nabla|s|\Big\|^2
		\leq \frac{n-1}{n}\big(\varepsilon^2+\frac{\kappa}{4}\big)+c_n\varepsilon^2+c_n\varepsilon\sqrt{\varepsilon^2+\frac{\kappa}{4}}.
	\end{align*}
	Since $\varepsilon$ can be chosen as any positive real number, we let $\varepsilon \rightarrow 0$, and then we obtain $$\lambda_1(X,g) \leq \frac{n-1}{4n}\ kappa.$$

	Next, let us prove the scalar curvature rigidity if the equality holds as follows. We will argue by contradiction. Suppose that there exists a positive constant $\delta>0$ such that the set $$X_\delta:=\{x\in X:\Sc_g(x)\geq -\kappa+\delta\}$$ is  a net of $X$, then there exists a discrete net $Y$ of $X$ and some $a>0$ such that
	$$\Sc_g(x)\geq -\kappa+\delta\text{ for any }x\in N_{a}(Y).$$
	Here, we have used the assumption of bounded geometry.
	
	Now given any $\varepsilon>0$, let $P_{\varepsilon^2}$ be the spectral projection to the spectrum $\leq\varepsilon^2$ and $V_{\varepsilon^2}$ the range of $P_{\varepsilon^2}$. Since $D$ is non-invertible,  we obtain that $V_{\varepsilon^2}$ is non-empty. Let us pick a spinor $s$ in $V_{\varepsilon^2}$ with $\|s\|=1$. Clearly we have $\|Ds\|\leq\varepsilon.$  By our assumption that  $\Sc_{g}\geq -\kappa+\delta$ on $N_a(Y)$, we obtain by the Lichnerowicz formula that
	$$\|\nabla s\|^2=\|D s\|^2-\int_X\frac{\Sc_g}{4}|s|^2\leq \varepsilon^2+\frac{\kappa}{4}-\frac{\delta}{4}\|s\|^2_{L^2(N_a(Y))}.$$
	Similarly, we deduce
	$$\Big\|\nabla|s|\Big\|^2
	\leq \frac{n-1}{n}\big(\varepsilon^2+\frac{\kappa}{4}\big)+c_n\varepsilon^2+c_n\varepsilon\sqrt{\varepsilon^2+\frac{\kappa}{4}}-\frac{(n-1)\delta}{4n}\|s\|^2_{L^2(N_a(Y))}.$$
	Assume that $\varepsilon<1$. By Proposition \ref{prop:uniqueCont}, there exists $C>0$ independent of $\varepsilon$ such that
	$$\|s\|_{L^2(N_a(Y))}\geq \frac 1 C\|s\|=\frac 1 C.$$
	Therefore, we see that
	$$\Big\|\nabla|s|\Big\|^2
	\leq\frac{(n-1)\kappa}{4n}-\frac{(n-1)\delta}{4nC^2}+\Big(	
	\frac{n-1}{n}\varepsilon^2+c_n\varepsilon^2+c_n\varepsilon\sqrt{\varepsilon^2+\frac{\kappa}{4}}\Big).$$
	By letting $\varepsilon\to 0$, we have
	$$\lambda_1(X,g)\leq  \frac{n-1}{4n}\kappa-\frac{(n-1)\delta}{4nC^2}<\frac{n-1}{4n}\kappa.$$
	This contradicts with the assumption that $\lambda_1(X,g)=\frac{n-1}{4n}\kappa$. This finishes the proof.
\end{proof}

We also list some topological conditions where $\ind(D)\in K_*(C^*(X))$ is non-zero.

\begin{proposition} \label{prop: novikov_nonzero_index}
	If $(X,g)$ is a geometrically contractible  Riemannian manifold and satisfies the Coarse Novikov Conjecture, then $\ind(D)$ is non-zero in $K_*(C^*(X))$. In particular, Theorem \ref{thm: geometrically_contracbile_spectrum} holds for $X$.
\end{proposition}
\begin{proof}
    We first note that the assumption of geometrically contractible implies that $X$ is spin, hence the Dirac operator $D$ is well-defined.
    
    Let $\mathcal N$ be a maximal $1$-discrete set of $X$, namely $d(x,y)\geq 1$ for any $x,y\in\mathcal N$, and $\mathcal N$ is maximal under inclusion of such sets. It follows from the coarse invariance of Roe algebra \cite[Theorem 5.1.15]{willett2020higher} that $K_*(C^*(\mathcal N))\cong K_*(C^*(X))$. Furthermore, since $X$ is geometrically contractible, for any $d>0$, the Rips complex $P_d(\mathcal N)$ is homotopic equivalence to $X$ in the sense of \cite[Theorem 6.4.16]{willett2020higher}, hence $K_*(C^*_L(P_d(\mathcal N))\cong K_*(C^*_L(X))$. Thus, the coarse Novikov conjecture for $X$ yields the injectivity of the index map
    $$K_*(C^*_L(X))\to K_*(C^*(X)).$$

    Under the assumption of bounded geometry, the $K$-theory $K_\ast(C^*_L(X))$ of the localization algebra $C^*_L(X)$ is naturally isomorphic to $K_*(X)$, the $K$-homology of $X$. Under this isomorphism, the local higher index of $D$ coincides with the $K$-homology class of $D$ (see \cite{Yulocalization,QiaoRoe} for the details).
    By \cite[Corollary 9.6.12]{willett2020higher}, we have the Poincar\'{e} duality
    $$K_n(X)\cong K_n(C^*_L(X))\cong \mathbb Z,$$
    which is generated by the local higher index of the Dirac operator $D$. As a result, $\ind(D)\ne 0$ in $K_n(C^*(X))$. 
\end{proof}

In particular, the above property holds for Riemannian manifolds with non-positive sectional curvature (see \cite{Yulocalization}*{Theorem 4.1}). Hence, a geometric version of Theorem \ref{thm:complete} is as follows.
\begin{corollary}
	Suppose that $(X^n,g)$ a Cartan--Hadamard manifold\footnote{A simply connected, complete Riemannian manifold $(X,g)$ is said to be a Cartan--Hadamard manifold if the sectional curvature is non-positive.} with bounded geometry. If $\Sc_g \geq -\kappa$, then
	\begin{equation*}
		\lambda_1( X, g) \leq \frac{n-1}{4n}\kappa.
	\end{equation*}
\end{corollary}
\begin{corollary}
	Suppose that $(X^n,g)$ has bounded geometry and is bi-Lipschitz equivalent to a Cartan--Hadamard manifold outside a compact set. If $\Sc_g \geq -\kappa$, then
	\begin{equation*}
		\lambda_1( X, g) \leq \frac{n-1}{4n}\kappa.
	\end{equation*}
\end{corollary}

In particular, the condition of the non-vanishing of the higher index holds for asymptotically hyperbolic manifolds. A complete Riemannian manifold $(X^2,g)$ is said to be asymptotically hyperbolic if it is conformally compact with the standard sphere $(\sph^{n-1},g_{\sph^{n-1}})$ as its conformal end, and there is a unique defining function $r$ in a collar neighborhood near infinity such that
$$g=\sinh^{-2}(r)\big(dr^2+g_{\sph^{n-1}}+\frac{r^n}{n} h+O(r^{n+1})\big)$$
where $h$ is a symmetric 2-tensor on $\sph^{n-1}$, and the asymptotic expression can be differentiated twice (see \cite{MR1879228} for the definition). The non-vanishing of the higher index holds for asymptotically hyperbolic manifolds follows from the pairing of the index of the Dirac operator with almost flat bundles (see  \cite[Chapter 11]{willett2020higher}).
Hence, we obtain that

	\begin{corollary}
		Suppose that $(X^n, g)$ is an asymptotically hyperbolic spin  manifold with scalar curvature $\Sc_g \geq -\kappa$, then
		\[\lambda_1({X}^n, g) \leq  \frac{n-1}{4n}\kappa.\]
	\end{corollary}
%

\section{Unique continuation Theorem  on Riemannian manifolds} \label{sec: unique_cont.}

Suppose that $(X^n, g)$ is  a complete Riemannian manifold with bounded geometry and 
$P$ is a second order elliptic differential operator on $X$ acting on a smooth bundle $E$ over $X$. Then, the elliptic operator theory shows that $P$ satisfies the G\r{a}rding's inequality. Namely, there exists constants $c,c'>0$ such that
\begin{equation}\label{eq:Garding}
	\langle P\sigma,\sigma\rangle\geq c\|\nabla \sigma\|^2-c'\|\sigma\|^2.
\end{equation}
 In this section, we will prove a quantitative unique continuation theorem as follows.

\begin{theorem}\label{thm:uniqueCont}
Suppose that $(X^n, g)$ is a complete Riemannian manifold with bounded geometry and   $Y$ is a discrete net of $X$ and $N_a(Y)$ the $a$-neighborhood of $Y$ for some $a>0$.  Let $E$ be a vector bundle over $X$ and $P$ a second order elliptic differential operator acting on $E$ satisfying the G\r{a}rding inequality  in line \eqref{eq:Garding}. If  $P_\lambda$ is the spectral projection of $P$ acting on $L^2(E)$ with spectrum $\leq \lambda$ and $V_\lambda$ is the range of $P_\lambda$, then there exists a constant $C_\lambda>0$ such that
	$$\|\sigma\|_{L^2(X)}\leq C_\lambda\|\sigma\|_{L^2(N_a(Y))}\text{ for any }\sigma\in V_\lambda,$$
    where $C_\lambda\leq c_1e^{c_2\lambda}$ for some $c_1,c_2>0$.
\end{theorem}
Theorem \ref{thm:uniqueCont} is essentially motivated by  \cite{MR1362555,MR3041662}. It plays an essential role in the proof of the main theorem regarding the scalar curvature rigidity/scalar curvature distribution.

\subsection{Local Carleman estimate}

In this subsection, we will prove a local Carleman estimate for elliptic differential operators on a discrete net in a complete manifold $(X^n, g)$.

Let $X\times\R_{\geq 0}$ be the product space of $X$ and the half real line. In the following proof, we will use the function $\varphi$ as a key ingredient in variant circumstances. To begin with, we consider the simple case when $y$ is a singleton in $X$ and give a detailed computation. Given any fixed point $y\in X$, we consider a function on $X\times\R_{\geq 0}$
$$\varphi(x,t)=e^{-t-d(x,y)^6},$$
whose derivatives along $X$ with order $\leq 5$ are small near $y$. 

{Given a fixed small $a>0$, let $\mathcal F$ be the space of smooth sections in $E$ over $X\times\R_{\geq 0}$ that are supported in $\{\varphi< a\}$ and vanish on $X\times\{0\}$. }Let
$$Q=-\frac{\partial^2}{\partial t^2}+P$$
be a differential operator that acts on $\mathcal F$. For any $h>0$, we define
$$Q_{\varphi}=e^{\varphi/h}\cdot Q\cdot e^{-\varphi/h}.$$

We first prove
\begin{lemma}\label{lemma:local}
	There exists $C_1,C_2>0$ such that for any $f\in\mathcal F$, we have
	
		\begin{equation*}
		\frac 1 h\|\frac{\partial f }{\partial t} \|^2+\frac 1 h\|\nabla f\|^2+\frac{1}{h^3}\|f\|^2  \leq 	C_1\|Q_\varphi f\|^2+\frac{C_2}{h}\int_{X\times\{0\}}|\frac{\partial f}{\partial t}|^2
	\end{equation*}
	for any $h>0$ sufficiently small.
\end{lemma}
\begin{proof}
	Let $A$ and $B$ be the self-adjoint and anti-self-adjoint parts of $Q_\varphi$ respectively, namely
	$$A=\frac{Q_\varphi+Q_\varphi^*}{2},\ B=\frac{Q_\varphi-Q_\varphi^*}{2}.$$
	A direct calculation  shows that
	$$A=Q-\frac{\dot\varphi^2}{h^2}+\mathscr R_1=-\frac{\partial^2}{\partial t^2}+P-\frac{\dot\varphi^2}{h^2}+\mathscr R_1,$$
	$$B=2\frac{\dot\varphi}{h}\frac{\partial}{\partial t}+\frac{\ddot \varphi}{h}+\mathscr R_2=\frac{\partial}{\partial t}\frac{\dot\varphi}{h} + \frac{\dot\varphi}{h}\frac{\partial}{\partial t} +\mathscr R_2.$$
	Here, we denote by $\dot\varphi, \ddot\varphi$  the derivatives of $\varphi$ with respect to $t\in\R$, and  $\mathscr R_1$ and $\mathscr R_2$ are the remainders given by the derivatives of $\varphi$ along $X$, which are small by the construction.
	
	Note that $Q_\varphi=A+B$, we have
	$$\|Q_\varphi f\|^2=\|Af\|^2+\|Bf\|^2+\langle Af,Bf\rangle+\langle Bf,Af\rangle.$$
	Since  $f$ is compactly supported within $X\times[0,a)$ and $f(x,0)=0$ for any $x \in X$, we have
	$$\langle Af,Bf\rangle=-\langle BAf,f\rangle$$
	 and 
	 $$\langle Bf,Af\rangle=\langle ABf,f\rangle-\langle Bf,\frac{\partial f}{\partial t}\rangle\Big|_0^a=\langle ABf,f\rangle+\int_{X\times \{0\}}\frac{2\dot\varphi}{h}|\frac{\partial f}{\partial t}|^2.$$
	 It follows that
	 $$\|Q_\varphi f\|^2-\int_{X\times \{0\}}\frac{2\dot\varphi}{h}|\frac{\partial}{\partial t}f|^2=\|Af\|^2+\|Bf\|^2+\langle[A,B]f,f\rangle.$$
	 Here $[A, B] = AB - BA$.
	 A direct computation shows that
	 \begin{align*}
	 	[A,B]=&[-\frac{\partial^2}{\partial t^2}-\frac{\dot\varphi^2}{h^2},\frac{\partial}{\partial t}\frac{\dot\varphi}{h} + \frac{\partial}{\partial t}\frac{\dot\varphi}{h}]+\mathscr R_3\\
	 	=&4\frac{\dot\varphi^2\ddot\varphi}{h^3}
	 	-\frac{\partial}{\partial t}\Big(2\frac{\ddot\varphi}{h}\frac{\partial}{\partial t}+\frac{\dddot\varphi}{h}\Big)-\Big(2\frac{\ddot\varphi}{h}\frac{\partial}{\partial t}+\frac{\dddot\varphi}{h}\Big)\frac{\partial}{\partial t}
	 	+\mathscr R_3.
	 \end{align*}
	 Here the remainder $\mathscr R_3$ is also small and will be ignored. By construction, we have $\dot\varphi^2\ddot\varphi>1/2$ and $\ddot\varphi>1/2$ on the support of $f$ if $a$ is small enough. Furthermore, by line \eqref{eq:Garding} and Cauchy--Schwarz inequality, we have
	 \begin{align*}
	 	\langle\frac{\dot\varphi^2}{h^3}f,f\rangle=&\frac 1 h\langle-\frac{\partial^2}{\partial t^2}f+Pf-Af+\mathscr R_1 f,f\rangle\\
	 	\geq& \frac 1 h\|\frac{\partial}{\partial t}f\|^2+\frac 1 h c\|\nabla f\|^2-\frac 1 h c'\|f\|^2\\
	 	&-\frac{h^{1/2}}{2}\|Af\|^2-\frac{1}{2h^{5/2}}\|f\|^2+\frac 1 h\langle\mathscr R_1 f,f\rangle,
	 \end{align*}
	 and
	 \begin{align*}
	 	&\langle-\frac{\partial}{\partial t}\Big(2\frac{\ddot\varphi}{h}\frac{\partial}{\partial t}+\frac{\dddot\varphi}{h}\Big)f-\Big(2\frac{\ddot\varphi}{h}\frac{\partial}{\partial t}+\frac{\dddot\varphi}{h}\Big)\frac{\partial}{\partial t} f,f\rangle\\
	 	=&\frac 1 h\langle{4\ddot\varphi}\frac{\partial}{\partial t}f,\frac{\partial}{\partial t}f\rangle +
	 	\frac 1 h\langle2\dddot\varphi f,\frac{\partial}{\partial t} f\rangle\\
	 	\geq &\frac 2 h\|\frac{\partial}{\partial t}f\|^2 -\frac{1}{h^{1/2}}\|\frac{\partial}{\partial t}f\|^2
	 	-\frac{1}{h^{3/2}}\langle |\dddot\varphi|^2f,f\rangle.
	 \end{align*}
	 Hence there exists $c_1>0$ such that
	 \begin{equation}\label{eq:[A,B]}
	 	\langle[A,B] f,f\rangle\geq c_1\Big(\frac 1 h\|\frac{\partial}{\partial t} f\|^2+\frac 1 h\|\nabla f\|^2+\frac{1}{h^3}\|f\|^2\Big)-c_1\sqrt h\|Af\|^2.
	 \end{equation}
	 This inequality indicates that $[A, B]$ is positive modulo $A$. Clearly $\|Bf\|^2\geq 0$.
	 This finishes the proof for $h$ sufficiently small.
\end{proof}

We remark that the key ingredient that proves Lemma \ref{lemma:local} is the non-negativity condition \eqref{eq:[A,B]}. This follows from the fact that the function $e^x$ has a positive second-order derivative, and the function $\varphi$ has a non-zero derivative along some direction. The estimate \eqref{eq:[A,B]} holds more generally if $\varphi$ satisfies H\"{o}rmander's condition \cite[Theorem 27.1.11]{HormanderIV}. See \cite[Section 3]{MR3041662}. The following lemma is directly from Lemma \ref{lemma:local} by substituting $f=e^{\varphi/h}g$.
\begin{lemma}\label{lemma:local2}
	There exists $C_1, C_2>0$ such that for any $g\in\mathcal F$, we have
	\begin{align*}
		& \int_{X\times\R_{\geq 0}}\Big(\frac 1 h |\frac{\partial g}{\partial t} |^2+\frac 1 h|\nabla g|^2+\frac{1}{h^3}|g|^2\Big)e^{2\varphi/h}\\
		\leq & C_1 \int_{X\times\R_{\geq 0}}e^{2\varphi/h}|Qg|^2+\frac{C_2}{h}\int_{X\times\{0\}}|\frac{\partial g}{\partial t}|^2.
	\end{align*}
		for any $h>0$ sufficiently small.
\end{lemma}

Now we consider the non-compact case. Let $Y$ be a discrete net of $X$.
	Given any fixed small $a>0$, let $\mathcal F_Y$ be the space of smooth sections $g$ of $E$ over $X\times\R_{\geq 0}$ that satisfy
	\begin{itemize}
		\item $g$ is supported in $N_a(Y)\times [0,a)$,
		\item $g|_{X\times \{0\}}=0$.
	\end{itemize}
Let $\varphi_Y$ be a function on $X\times\R_{\geq 0}$ defined by
$$\varphi_Y(x,t)=e^{-t-d(x,y)^6}$$
on $B_{2a}(y)\times[0,2a]$ for any $y\in Y$. We assume that $a$ is small enough  so that the $4a$-neighborhoods of points in $Y$ are disjoint in Definition \ref{def:net}. The value of $\varphi_Y$ on the rest of points in $X$ can be arbitrary.

As it is pointed out in line \eqref{eq:[A,B]}, the essential part for the proof of Lemma \ref{lemma:local} is the non-negativity condition \eqref{eq:[A,B]} on the support of $f$. Here, we note that if $f\in\mathcal F_Y$, line \eqref{eq:[A,B]}  holds for $\varphi_Y$ as well. Thus, Lemma \ref{lemma:local2} still holds for the non-compact case.
\begin{lemma}\label{lemma:local3}
	There exists $C_1,C_2>0$ such that for any $g\in\mathcal F_Y$, we have
	\begin{equation}\label{eq:local3}
	\begin{split}
			& \int_{X\times\R_{\geq 0}}\Big(\frac 1 h|\frac{\partial g}{\partial t} |^2+\frac 1 h|\nabla g|^2+\frac{1}{h^3}|g|^2\Big)e^{2\varphi_Y/h}\\
			\leq &C_1\int_{X\times\R_{\geq 0}}e^{2\varphi_Y/h}|Qg|^2+\frac{C_2}{h}\int_{X\times\{0\}}|\frac{\partial g}{\partial t}|^2\\	
	\end{split}
	\end{equation}
		for any $h>0$ sufficiently small.
\end{lemma}

Moreover,  we will consider another type of function $\varphi$ along directions in $X$. Given a discrete net $Y$ of $X$, let $Z=\{Z_i\}$ be a collection of pieces of oriented hypersurfaces, where each piece is located near a point of $Y$. We fix a small number $t_0>0$, and points $z_i\in Z_i$. Pick smooth functions $v_i$ supported near $Z_i$ such that $|\nabla v_i|=1$, $Z_i$ is the level set $\{v_i=0\}$, and $\nabla v_i$ is pointing outward from $Z_i$. We define $\varphi_Z$ on $X\times\R_{\geq 0}$ as
$$\varphi_Z(x,t)=-v_i-d((x,t),(z_i,t_0))^6$$
near each $Z_i$. The value of $\varphi$ away from $Z_i$ is arbitrary.

Let $\mathcal F_Z$ be the collections of smooth sections of $E$ over $X\times\R_{\geq 0}$ that are supported in a small neighborhood of  $Z\times\{t_0\}$. The same proof of Lemma \ref{lemma:local3} applies to the function $\varphi_Z$.
\begin{lemma}\label{lemma:localZ}
	There exists $C_1>0$ such that for any $g\in\mathcal F_Z$, we have
	\begin{equation} \label{eq:localZ}
		\int_{X\times\R_{\geq 0}}\Big(\frac 1 h|\frac{\partial g}{\partial t}g|^2+\frac 1 h|\nabla g|^2+\frac{1}{h^3}|g|^2 \Big)e^{2\varphi_Z/h} \leq C_1\int_{X\times\R_{\geq 0}}e^{2\varphi_Z/h}|Qg|^2
	\end{equation}
	for any $h>0$ sufficiently small.
\end{lemma}
Note that the calculation in Lemma \ref{lemma:local} applies to the function $\varphi_Z$ if we replace the $t$-direction derivative by the $\nabla v_i$-directions. Thus,  Lemma \ref{lemma:localZ} follows from a similar calculation. We also emphasize that, as $g$ is supported near $Z$, we only need the value of the function $\varphi_Z$ near $Z$ in the proof. Since $g$ vanishes away from $Z$, the boundary term, namely the second term in the right-hand side of line \eqref{eq:local3}, does not appear in line \eqref{eq:localZ}.
\subsection{Interpolation and unique continuation}

In this subsection, we will first prove an interpolation inequality for sections over $X\times\R_{\geq 0}$ and then derive the unique continuation theorem at the lower spectrum of elliptic operators from a net.

We begin with some elementary inequalities that deduce interpolation inequality from a Carleman estimate.
\begin{lemma}\label{lemma:interpolation}
	Let $\alpha,\beta,\gamma$ be positive numbers with $\alpha\leq A\beta$ for some $A>0$. If there exist $p,q>0$ and $h_0>0$ such that
	$$\alpha\leq e^{-p/h}\beta+e^{q/h}\gamma$$
	for any $h\in(0,h_0)$, then there exist $C>0$ and $\nu\in(0,1)$ that only depends on $A,p,q,h_0$ such that
	$$\alpha\leq C\beta^\nu\gamma^{1-\nu}.$$
\end{lemma}
\begin{proof}
	We set $\nu=\frac{q}{p+q}$ and  define the function $F(h)=e^{-p/h}\beta+e^{q/h}\gamma$ on $\R^+$. A direct calculation shows that  $F$ attains its unique minimum at the point
	$$h=h_*=\frac{\ln(p\beta)-\ln(q\gamma)}{p+q}, $$ and the minimum value is  $$F(h_*)=(p+q)p^{-\frac{p}{p+q}}q^{-\frac{q}{p+q}}\cdot \beta^\nu\gamma^{1-\nu}.$$
	
	Let us consider the following cases.
	
	\begin{itemize}
		\item If $h_*\leq h_0$, then the desired inequality follows directly.

		\item Assume that $h_*\geq h_0$.
		
		\begin{itemize}
			\item 	If $\beta\leq\gamma$, then we have obviously 
			$$\alpha\leq A\beta\leq A\beta^\nu\gamma^{1-\nu}.$$
			
			\item  If $\gamma\leq\beta$, then by the monotonicity of $F$ on $(0,h_0)$, we have
			\begin{align*}
				\alpha&\leq F(h_0)=e^{-p/h_0}\beta+e^{q/h_0}\gamma\\
				&\leq e^{-p/h_*}\beta+e^{q/h_0}\gamma=q^{\frac{p}{p+q}}p^{-\frac{p}{p+q}}\beta^\nu\gamma^{1-\nu}+e^{q/h_0}\gamma\\
				&\leq \Big(q^{\frac{p}{p+q}}p^{-\frac{p}{p+q}}+e^{q/h_0}\Big)\beta^\nu\gamma^{1-\nu},
			\end{align*}
			where the last inequality follows from $\gamma\leq\beta$. 
		\end{itemize}
	\end{itemize}
To summarize, we have shown that $\alpha\leq C\beta^\nu\gamma^{1-\nu}$ by setting
$$C=\max\Big\{(p+q)p^{-\frac{p}{p+q}}q^{-\frac{q}{p+q}},A,q^{\frac{p}{p+q}}p^{-\frac{p}{p+q}}+e^{q/h_0}\Big\}.$$

\end{proof}
\begin{lemma}\label{lemma:iteratedInterpolation}
	Suppose that  $\alpha_i>0$ for $i=0,1,\cdots,N$ and $\beta,\gamma$ are positive numbers with $\alpha_i\leq \beta$ for any $i$. If there exists $\nu\in(0,1)$ and $C\geq 1$ such that 
	$$\alpha_{k+1}\leq C\beta^\nu(\alpha_{k}+\gamma)^{1-\nu},\ k = 0, \cdots, N-1$$
	 then,
	$$\alpha_N\leq C'\beta^\mu(\alpha_{0}+\gamma)^{1-\mu},$$
	where $\mu=1-(1-\nu)^N$ and $C'=(2C)^{1+(1-\nu)+\cdots+(1-\nu)^{N-1}}$.
\end{lemma}
\begin{proof}
	If $\gamma\geq\beta$, then $\alpha_0+\gamma\geq\beta$, then
	$$\alpha_N\leq \beta\leq\beta^\mu(\alpha_0+\gamma)^{1-\mu}.$$
	Now we assume $\gamma\leq \beta$, then we obtain that
	$$\frac{\gamma}{\beta}\leq \frac{\gamma^{1-\nu}}{\beta^{1-\nu}}\leq \frac{(\alpha_k+\gamma)^{1-\nu}}{\beta^{1-\nu}}\leq C\frac{(\alpha_k+\gamma)^{1-\nu}}{\beta^{1-\nu}}$$
	for any $k$.  Moreover, the assumption implies that
	$$\frac{\alpha_{k+1}}{\beta}\leq C\frac{(\alpha_k+\gamma)^{1-\nu}}{\beta^{1-\nu}}.$$
	Hence, we reach
	$$\frac{\alpha_{k+1}+\gamma}{\beta}\leq 2 C\frac{(\alpha_k+\gamma)^{1-\nu}}{\beta^{1-\nu}}.$$
	Therefore,
	$$\frac{\alpha_N}{\beta}\leq 2C\frac{(\alpha_{N-1}+\gamma)^{1-\nu}}{\beta^{1-\nu}}\leq \cdots\leq C'\frac{(\alpha_0+\gamma)^{(1-\nu)^N}}{\beta^{(1-\nu)^N}}.$$
	Equivalently,
	$$\alpha_N\leq C' \beta^\mu(\alpha_0+\gamma)^{1-\mu}.$$
	This finishes the proof.	
\end{proof}

We start from the following lemma by applying the construction in Lemma \ref{lemma:local3} first.
\begin{lemma}\label{lemma:interpolation1}
	Suppose that $X$ is a complete Riemannian manifold, and $Y$ is a discrete net in $X$, and $a$ is a small positive number. Given small positive numbers  $\tau \ll t_0<T\ll a$ and $a_1\ll a$, there exists $C>0$ and $\nu\in(0,1)$ such that,  for any smooth section $\sigma$ of $E$ over $X\times\R_{\geq 0}$, we have
		\begin{equation*}
		\|\sigma\|_{H_1(N_{a_1}(Y)\times N_\tau(t_0))}\leq C \|\sigma\|_{H_1(X\times[0,T])}^\nu\Big(\|Q\sigma\|_{L^2(X\times[0,T])}+\Big\|\frac{\partial \sigma}{\partial t}\Big\|_{L^2(N_a(Y)\times\{0\})}\Big)^{1-\nu}.
	\end{equation*}
\end{lemma}
\begin{proof}
	For any $b>0$, we define
	$$\Omega_b=\{(x,t)\in N_a(Y)\times[0,a):\varphi_Y(x,t)\geq b\}.$$ Let $b_1<b_3<0$ such that $$H_1(N_{a_1}(Y)\times N_\tau(t_0))\subset \Omega_{b_1}\subset \Omega_{b_3}\subset X\times[0,T].$$
	We shall prove that there exists a constant $C>0$,
		\begin{equation}
		\|\sigma\|_{H_1(\Omega_{b_1})}\leq C \|\sigma\|_{H_1(\Omega_{b_3})}^\nu\Big(\|Q\sigma\|_{L^2(\Omega_{b_3})}+\Big\|\frac{\partial}{\partial t}\sigma\Big\|_{L^2(N_a(Y)\times\{0\})}\Big)^{1-\nu}.
	\end{equation}
	
	Let $\rho$ be a smooth non-increasing function on $\R$ such that $\rho(s)=1$ if $s\leq b_1$, and $\rho(s)=0$ if $s\geq b_3$.
%
	Set $\chi=\rho\circ\varphi_Y$. It is straightforward that $\nabla\chi$ is only supported on $\Omega_{b_3}-\Omega_{b_1}$. We fix $b_2\in({b_1},{b_3})$ such that $\rho(s)=1/2$. Let $g=\chi\sigma$, which lies in $\mathcal F_Y$ by assumption. 
	
	First, we consider the right-hand side of line \eqref{eq:local3}. Since $\frac{\partial}{\partial t}( \chi\sigma)=\chi\frac{\partial}{\partial t} \sigma+\frac{\partial\chi}{\partial t} \sigma$ and $\frac{\partial\chi}{\partial t}$ is only supported on $\Omega_{b_3}-\Omega_{b_1}$, we have
	\begin{equation*}
		\int_{X\times\R_{\geq 0}}\frac 1 he^{2\varphi_Y/h}\Big|\frac{\partial}{\partial t} (\chi\sigma)\Big|^2\geq \frac{1}{8h}e^{2b_2/h}\int_{\Omega_{b_2}}\Big|\frac{\partial}{\partial t} \sigma\Big|^2-\|\nabla\chi\|_\infty\frac 1 he^{2{b_1}/h}\int_{\Omega_{b_3}-\Omega_{b_1}}|\sigma|^2.
	\end{equation*}
	Similarly,
		\begin{equation*}
		\int_{X\times\R_{\geq 0}}\frac 1 he^{2\varphi_Y/h}|\nabla(\chi\sigma)|^2\geq \frac{1}{8h}e^{2b_2/h}\int_{\Omega_{b_2}}|\nabla \sigma|^2-\|\nabla\chi\|_\infty\frac 1 he^{2{b_1}/h}\int_{\Omega_{b_3}-\Omega_{b_1}}|\sigma|^2.
	\end{equation*}
	It is also clear that
	\begin{equation*}
		\int_{X\times\R_{\geq 0}}\frac{1}{h^3}e^{2\varphi_Y/h}|\chi\sigma|^2\geq \frac{1}{4h^3}e^{2b_2/h}\int_{\Omega_{b_2}}|\sigma|^2.
	\end{equation*}
	
	Secondly,  we consider the left-hand side of line \eqref{eq:local3}. It is clear that
	\begin{equation*}
		\frac{1}{h}\int_{X\times\{0\}}\Big|\frac{\partial}{\partial t}(\chi\sigma)\Big|^2\leq \frac{1}{h}\int_{N_{b_3}(Y)\times\{0\}}\Big|\frac{\partial}{\partial t}\sigma\Big|^2.
	\end{equation*}
	We note that $Q(\chi\sigma)=\chi(Q\sigma)+[Q,\chi]\sigma$, where $[Q,\chi]$ is a first-order differential operator that is supported on $\Omega_{b_3}-\Omega_{b_1}$. Therefore, there exists $c_1>0$ such that
	\begin{equation*}
		\begin{split}
			\int_{X\times\R_{\geq 0}}e^{2\varphi_Y/h}|Q(\chi\sigma)|^2\leq& \frac 1 2\int_{X\times\R_{\geq 0}}e^{2\varphi_Y/h}|\chi Q\sigma|^2+\int_{\Omega_{b_3}-\Omega_{{b_1}}}e^{2\varphi_Y/h}|[Q,\chi]\sigma|^2\\
			\leq &\frac1 2\|Q\sigma\|_{L^2(\Omega_{b_3})}+c_1e^{2{b_1}/h}\|\sigma\|^2_{H^1(\Omega_{b_3})}.
		\end{split}
	\end{equation*}
	Combining all the inequalities above, we reach that there exists $c_2>0$ such that
	\begin{equation*}
		e^{2b_2/h}\|\sigma\|_{H^1(\Omega_{b_2})}^2\leq c_2e^{2{b_1}/h}\|\sigma\|_{H^1(\Omega_{b_3})}^2+c_2\Big(\|Q\sigma\|_{L^2(\Omega_{b_3})}+\Big\|\frac{\partial}{\partial t}\sigma\Big\|_{L^2(N_a(Y)\times\{0\})}\Big).
	\end{equation*}
	Thus,
		\begin{equation*}
		\|\sigma\|_{H^1(\Omega_{{b_1}})}^2\leq c_2e^{2({b_1}-b_2)/h}\|\sigma\|_{H^1(\Omega_{b_3})}^2+c_2e^{-2b_2/h}\Big(\|Q\sigma\|_{L^2(\Omega_{b_3})}+\Big\|\frac{\partial}{\partial t}\sigma\Big\|_{L^2(N_a(Y)\times\{0\})}\Big).
	\end{equation*}
	In particular, there exists $h_0>0$ such that the above inequality holds uniformly for any $h\in(0,h_0)$. We emphasize that here ${b_1}-b_2<0$ and $-b_2>0$. Clearly $\|\sigma\|_{H^1(\Omega_{{b_1}})}\leq \|\sigma\|_{H^1(\Omega_{{b_3}})}$. This finishes the proof by applying Lemma \ref{lemma:interpolation}.
\end{proof}

Lemma \ref{lemma:interpolation1} shows that the $H^1$-norm of $\sigma$ on $H^1(N_{a_1}(Y)\times N_\tau(t_0))$ is bounded in the sense of interpolation. By the assumption in Definition \ref{def:net}, the $r_2$-neighborhood of $Y$ covers the entire $X$ for some $r_2>0$. We shall prove that the $H^1$-norm of $\sigma$ on $X\times N_\tau(t_0)$ is also bounded in the sense of interpolation, by increasing the radius $a_1$.

\begin{proposition}\label{prop:interpolation}
		Let $Y$ be a discrete net in $X$ and $a$ a small positive number. Given small positive numbers $\tau\ll t_0<T\ll a$, there exists $C>0$ and $\nu\in(0,1)$ such that for any smooth section $\sigma$ of $E$ over $X\times\R_{\geq 0}$, we have
	\begin{equation}\label{eq:interpolation}
	\begin{split}
			&\|\sigma\|_{H_1(X\times N_\tau(t_0))}\\\leq &C \|\sigma\|_{H_1(X\times[0,T])}^\nu\Big(\|Q\sigma\|_{L^2(X\times[0,T])}+\Big\|\frac{\partial}{\partial t}\sigma\Big\|_{L^2(N_a(Y)\times\{0\})}\Big)^{1-\nu}.
	\end{split}
	\end{equation}
\end{proposition}
\begin{proof}
	We shall prove that there exists $\varepsilon>0$ such that
		\begin{equation}\label{eq:interpolationH1}
		\begin{split}
			&\|\sigma\|_{H_1(N_{a_1+\varepsilon}(Y)\times N_\tau(t_0))}\\\leq &C_1 \|\sigma\|_{H_1(X\times[0,T])}^{\nu_1}\Big(\|Q\sigma\|_{L^2(X\times[0,T])}+\|\sigma\|_{H_1(N_{a_1}(Y)\times N_\tau(t_0))}\Big)^{1-\nu_1}.
		\end{split}
	\end{equation}
	for some $C_1>0$ and $\nu_1>0$. By Definition \ref{def:net}, there are only finitely many steps to exhaust $X$ from $Y$ by increasing $\varepsilon$ of the neighborhood of $Y$. Since $X$ has bounded geometry, Lemma \ref{lemma:interpolation1}, Lemma \ref{lemma:iteratedInterpolation} and line \eqref{eq:interpolationH1} together implies line \eqref{eq:interpolation}.
		\begin{figure}[h]
		\begin{tikzpicture}[scale=1]\label{fig}
			\node at (-7.3,0){$Z_i$};
			\node at (-5.6,-2.2){$\varphi_i=b_2$};
			\node at (-5.6,-2.2+.6){$\varphi_i=b_1$};
			\fill[gray,opacity=0.3] (3.943,-0.67) arc(-9.5:189.5:4) -- (-3,0) arc (180:0:3) -- cycle;
			\node at (0,3.5)[fill=white]{$\widehat\Omega_{\text{out}}$};
			\draw (0,0) circle[radius=0.5];
			\draw[very thick] (-7,0) -- (7,0);
			\fill (0,0) circle[radius=0.05];
			\draw (0,0) circle[radius=3];
			\draw (0,0) circle[radius=4];
			\draw[thick] (-5,-2) .. controls (0,1.5) .. (5,-2);
			\draw[thick] (-5,-2+.6) .. controls (0,1.5+.6) .. (5,-2+.6);
			\node at (0,0)[fill=white]{$z_i$};
			\path[clip] (3.943,-0.67) arc(-9.5:-170.5:4) -- (-3,0) arc (-180:0:3) -- cycle;
			\foreach \y in {0,-.2,...,-4}
			\draw (-4,\y) -- (4,\y);
			\node at (0,-3.5)[fill=white]{$\widehat\Omega_{\text{in}}$};
		\end{tikzpicture}
	\end{figure}
	
	We shall prove line \eqref{eq:interpolationH1} by applying Lemma \ref{lemma:localZ} with carefully chosen $Z$ and $\varphi_Z$. Once chosen, the rest of the proof is completely similar to the proof of Lemma \ref{lemma:interpolation}. 
	Given $N_{a_1}(Y)=\cup_i N_{a_1}(y_i)$, let $z_i$ be a point on $\partial N_{a_1}(y_i)$, $Z_i$ a tiny piece of $\partial N_{a_1}(y_i)$ near $z_i$, and $Z=\cup_iZ_i$. 
	Pick smooth functions $v_i$ supported near $Z_i$ such that $|\nabla v_i|=1$, $Z_i$ is the level set $\{v_i=0\}$, and $\nabla v_i$ is pointing outward from $Z_i\subset \partial N_{a_1}(Y)$. We consider the function $\varphi_Z$ on $X\times\R_{\geq 0}$ such that
	$$\varphi_Z(x,t)=-v_i-d((x,t),(z_i,t_0))^6$$
	near each $Z_i$. 
	
	Let $\chi$ be a smooth cut-off function that is equal to $1$ on $N_{\varepsilon_1}(Z\times\{t_0\})$ and equal to $0$ outside $N_{\varepsilon_2}(Z\times\{t_0\})$. Denote $\widehat\Omega=N_{\varepsilon_2}(Z\times\{t_0\})-N_{\varepsilon_1}(Z\times\{t_0\})$, which contains the supported of $\nabla\chi$.
	Similar to the proof of Lemma \ref{lemma:interpolation}, we define
	$$\Omega_b=\{(x,t)\in N_{\varepsilon_1}(Z\times\{t_0\}):\varphi_Z(x,t)\geq b\}.$$
	By construction of $\varphi_Z$, we have $\varphi_Z(z_i,t_0)=0$, and there exists $b_1<0$ such that $\Omega_{b_1}\cap\widehat\Omega$ is contained inside $N_{a_1}(Y)\times N_\tau(t_0)$. Pick $b_2,b_3$ with $b_1<b_2<0<b_3$ and $\varepsilon>0$ such that $N_\varepsilon(Z)\times N_\tau(t_0)\subset \Omega_{b_2}$ and $N_{\varepsilon_2}(Z\times\{t_0\})\subset (\Omega_{b_3})^c$. 
	
	Now we consider the section $\chi\sigma$, which lies in $\mathcal F_Z$ by assumption, hence satisfies the inequality in Lemma \ref{lemma:localZ}. Similar to the computation in the proof of Lemma \ref{lemma:interpolation1}, the right-hand side of line \eqref{eq:localZ} satisfies the following:
	\begin{align*}
	\int_{X\times\R_{\geq 0}}\frac 1 he^{2\varphi_Z/h}\Big(|\frac{\partial}{\partial t} (\chi\sigma)|^2+|\nabla (\chi\sigma)|^2\Big)\geq&\int_{N_\varepsilon(Z)\times N_\tau(t_0)}\frac 1 he^{2\varphi_Z/h}\Big(|\frac{\partial}{\partial t} (\chi\sigma)|^2+|\nabla (\chi\sigma)|^2\Big)\\
	\geq &\frac 1 he^{2b_2/h}\int_{N_\varepsilon(Z)\times N_\tau(t_0)}\Big(|\frac{\partial}{\partial t} (\chi\sigma)|^2+|\nabla (\chi\sigma)|^2\Big),
	\end{align*}
	and
	\begin{align*}
		\int_{X\times\R_{\geq 0}}\frac{1}{h^3}e^{2\varphi_Z/h}|\chi\sigma|^2\geq \frac{1}{h^3}e^{2b_2/h}\int_{N_\varepsilon(Z)\times N_\tau(t_0)}|\sigma|^2.
	\end{align*}	
	For the left-hand side of line \eqref{eq:localZ}, we still notice that 
	$$Q(\chi\sigma)=\chi Q\sigma+[Q,\chi]\sigma,$$
	where $[Q,\chi]$ is a first-order differential operator that is supported only on $\widehat\Omega$. Write $\widehat\Omega=\widehat\Omega_{\text{in}}\cup \widehat\Omega_{\text{out}}$, where $\widehat\Omega_{\text{in}}\coloneqq \Omega_{b_1}\cap\widehat\Omega$ and $\widehat\Omega_{\text{out}}$ is the complement of $\widehat\Omega_{\text{in}}$. We note that by construction, $\widehat\Omega_{\text{in}}$ is contained in $N_{a_1}(Y)\times N_\tau(t_0)$, while on  $\widehat\Omega_{\text{out}}$ we have $\varphi_Z\leq b_1$. Therefore, there exists $c_1>0$ such that
	\begin{align*}
		&\int_{X\times\R_{\geq 0}}e^{2\varphi_Z/h}|Qg|^2\leq \frac 1 2\int_{N_{\varepsilon_1}(Z\times\{t_0\})}e^{2\varphi_Z/h}|\chi Q\sigma|^2+\int_{\widehat\Omega}e^{2\varphi_Z/h}|[Q,\chi]\sigma|^2\\
	\leq &\frac1 2e^{2b_3/h}\int_{L^2(N_{\varepsilon_1}(Z\times\{t_0\}))}|Q\sigma|^2+\int_{\widehat\Omega_{\text{in}}}e^{2\varphi_Z/h}|[Q,\chi]\sigma|^2+\int_{\widehat\Omega_{\text{out}}}e^{2\varphi_Z/h}|[Q,\chi]\sigma|^2\\
	\leq &\frac1 2e^{2b_3/h}\|Q\sigma\|^2_{L^2(X\times[0,T])}+c_1e^{2b_3/h}\|\sigma\|^2_{H_1(N_{a_1}(Y)\times N_\tau(t_0))}+c_1e^{2b_1/h}\|\sigma\|^2_{H^1(X\times[0,T])}
	\end{align*}
	Therefore, by Lemma \ref{lemma:localZ}, there is $c_2>0$ such that
	\begin{align*}
		&e^{2b_2/h}\|\sigma\|_{H^1(N_\varepsilon(Z)\times N_\tau(t_0))}^2\\ \leq& c_2\Big(e^{2b_3/h}\|Q\sigma\|^2_{L^2(X\times[0,T])}+e^{2b_3/h}\|\sigma\|^2_{H_1(N_{a_1}(Y)\times N_\tau(t_0))}+e^{2b_1/h}\|\sigma\|^2_{H^1(X\times[0,T])}\Big).
	\end{align*}
	Equivalently,
	\begin{align*}
		&\|\sigma\|_{H^1(N_\varepsilon(Z)\times N_\tau(t_0))}^2\\ \leq& c_2e^{2(b_1-b_2)/h}\|\sigma\|^2_{H^1(X\times[0,T])}+c_2e^{2(b_3-b_2)/h}\Big(\|Q\sigma\|^2_{L^2(X\times[0,T])}+\|\sigma\|^2_{H_1(N_{a_1}(Y)\times N_\tau(t_0))}\Big).
	\end{align*}
	We note that $b_1-b_2<0$ and $b_3-b_2>0$. It follows together with Lemma \ref{lemma:interpolation} that
	\begin{align*}
		\|\sigma\|_{H^1(N_\varepsilon(Z)\times N_\tau(t_0))}\leq c_3\|\sigma\|^{\nu_1}_{H^1(X\times[0,T])}\Big(\|Q\sigma\|^2_{L^2(X\times[0,T])}+\|\sigma\|^2_{H_1(N_{a_1}(Y)\times N_\tau(t_0))}\Big)^{1-\nu_1}.
	\end{align*}
	for some $c_3>0$ and $\nu_1>0$. Note that as $X$ has bounded geometry $N_{a_1+\varepsilon}(Y)$ is covered by at most $N$ sets, which are of the form $N_\varepsilon(Z)$ for some $Z\subset\partial N_{a_1}(Y)$. This finishes the proof of line \eqref{eq:interpolationH1} with $C_3=Nc_1$, hence completes the proof of line \eqref{eq:interpolation} by the discussion at the beginning.	
\end{proof}

Finally, we are ready to prove Theorem \ref{thm:uniqueCont}.

\begin{proof}
 The G\r{a}rding inequality implies that
	\begin{equation}
		\langle P\psi,\psi\rangle\geq c\|\nabla\psi\|^2-c'\|\psi\|^2,
	\end{equation}
	for any $L^2$-section $\psi$. Thus, we obtain that $P+c'\geq 0$. 
	Without loss of generality, we may assume that $P\geq 0$.	
	
	Given $\sigma\in V_\lambda$, we define
	$$F_t=\frac{\sinh(t\sqrt{P})}{\sqrt P}\sigma=\frac{e^{t\sqrt{P}}-e^{-t\sqrt{P}}}{2\sqrt{P}}\sigma=\sum_{n=0}^{\infty}\frac{t^{2n+1}}{(2n+1)!}P^n\sigma.$$
	It is clear from the definition that
	$$\frac{\partial}{\partial t}F_t\Big|_{t=0}=\sigma,\text{ and }QF_t=\Big(-\frac{\partial^2}{\partial t^2}+P\Big)F_t=0.$$
	Together with Proposition \ref{prop:interpolation}, we obtain that
	\begin{equation}\label{eq:interpolationFt}
			\|F_t\|_{H_1(X\times (t_0-\tau,t_0+\tau))}\leq C \|F_t\|_{H_1(X\times[0,T])}^\nu\|\sigma\|_{L^2(N_a(Y))}^{1-\nu}.
	\end{equation}
	
	Moreover, the construction of $F_t$ directly implies that, for any $t\in[0,T]$,
	$$\|F_t\|\leq \frac{\sinh(t\sqrt\lambda)}{\sqrt\lambda}\|\sigma\|,~\langle PF_t,F_t\rangle\leq \sqrt{\lambda}\sinh(t\sqrt{\lambda})\|\sigma\|,$$
	and
	$$\|\sigma\|\leq \Big\|\frac{\partial}{\partial t}F_t\Big\|\leq \cosh(t\sqrt{\lambda})\|\sigma\|.$$
	Therefore,
	\begin{align*}
		\|F_t\|_{H^1(M\times[0,T])}^2\leq&\int_0^T\Big((1+\frac{c'}{c})\|F_t\|^2+\Big\|\frac{\partial}{\partial t}F_t\Big\|^2+\frac 1 c\langle PF_t,F_t\rangle\Big) dt\\
		\leq &T\Big((1+\frac{c'}{c}\frac{\sinh^2(T\sqrt\lambda)}{\lambda}+\cosh^2(T\sqrt{\lambda}))+\frac 1 c\lambda\sinh^2(T\lambda)\Big)\|\sigma\|^2\\
		=&C_\lambda^{\frac{2(1-\nu)}{\nu}}\big(\frac{2\tau}{C^2}\big)^{\frac{1}{\nu}}\|\sigma\|^2
	\end{align*}
	Here, $$C_\lambda=\big(\frac{C^2}{2\tau}\big)^{\frac{1}{2(1-\nu)}}T^{\frac{\nu}{2(1-\nu)}}\Big((1+\frac{c'}{c}\frac{\sinh^2(T\sqrt\lambda)}{\lambda}+\cosh^2(T\sqrt{\lambda}))+\frac 1 c\lambda\sinh^2(T\lambda)\Big)^{\frac{\nu}{2(1-\nu)}}.$$
	By the G\r{a}rding inequality, we have that
	$$\|F_t\|_{H_1(X\times (t_0-\tau,t_0+\tau))}\geq \int_{t_0-\tau}^{t_0+\tau}\Big\|\frac{\partial}{\partial t}F_t\Big\|^2 dt\geq 2\tau \|\sigma\|^2.$$
	Thus we obtain from line \eqref{eq:interpolationFt} that
	{
	$$2\tau \|\sigma\|^2\leq C^2 \Big(C_\lambda^{\frac{2(1-\nu)}{\nu}}\big(\frac{2\tau}{C^2}\big)^{\frac{1}{\nu}}\|\sigma\|^2\Big)^\nu\|\sigma\|_{L^2(N_a(Y))}^{2-2\nu}.$$	}
	A direct simplification indicates that 
	$$\|\sigma\|_{L^2(X)}\leq C_\lambda\|\sigma\|_{L^2(N_a(Y))}.$$

\end{proof}



\end{document}